\UseRawInputEncoding
\documentclass[12pt]{amsart}
\usepackage{}

\usepackage{cancel}
\usepackage{amsmath}
\usepackage{amsfonts}
\usepackage{amssymb}
\usepackage[all]{xy}           %xypic macro for latex2.09

\usepackage{mathrsfs}
\usepackage{bbding}
\usepackage{txfonts}
\usepackage{amscd}

\usepackage[shortlabels]{enumitem}
\usepackage{ifpdf}
\ifpdf
  \usepackage[colorlinks,final,backref=page,hyperindex]{hyperref}
\else
  \usepackage[colorlinks,final,backref=page,hyperindex,hypertex]{hyperref}
\fi
\usepackage{tikz}
\usepackage[active]{srcltx}

\makeatletter

\newtheorem{thm}{Theorem}[section]
\newtheorem{lem}[thm]{Lemma}

\newtheorem{pro}[thm]{Proposition}
\newtheorem{ex}[thm]{Example}
\newtheorem{rmk}[thm]{Remark}
\newtheorem{defi}[thm]{Definition}
\newtheorem{que}[thm]{Question}

\newcommand {\emptycomment}[1]{}

\newcommand{\lon }{\,\rightarrow\,}
\newcommand{\be }{\begin{equation}}
\newcommand{\ee }{\end{equation}}

\newcommand{\g}{\mathfrak g}
\newcommand{\h}{\mathfrak h}

\newcommand{\huaL}{\mathcal{L}}

\newcommand{\huaW}{\mathcal{W}}

\newcommand{\huaD}{\mathcal{D}}

\newcommand{\huaK}{\mathcal{K}}

\newcommand{\frkk}{\mathfrak k}

\newcommand{\frkz}{\mathfrak z}

\newcommand{\dM}{\mathrm{d}}

\newcommand{\Id}{{\rm{Id}}}

\newcommand{\br}[1]{   [ \cdot,    \cdot  ]   }

\newcommand{\Hom}{\mathrm{Hom}}

\newcommand{\LD}{{\rm LieDer }}
\newcommand{\LtoD}{{\rm Lie2Der }}
\newcommand{\Der}{\mathrm{Der}}
\newcommand{\NR}{\mathrm{NR}}

\newcommand{\Out}{\mathrm{Out}}

\newcommand{\gl}{\mathfrak {gl}}

\newcommand{\ad}{\mathrm{ad}}

\begin{document}

\title{Non-abelian Extensions of Lie algebras with derivations}

\author{Jun Jiang}
\address{Department of Mathematics, Jilin University, Changchun 130012, Jilin, China}
\email{junjiang@jlu.edu.cn}

\author{Kanghe Xu}
\address{Department of Mathematics, Jilin University, Changchun 130012, Jilin, China}
\email{xukh24@mails.jlu.edu.cn}

%\date{\today}

\begin{abstract}
In this paper, we investigate non-abelian extensions of Lie algebras with derivations from several different perspectives. We show that the theory of non-abelian extensions of a Lie algebra with a derivation can be characterized by means of the second non-abelian cohomology, the Deligne groupoid, the homotopy category of strict Lie $2$-algebras with strict derivations, and the notion of a $(\mathfrak{g}, D)$-kernel, respectively. Moreover, within this unified framework, we address the following existence problem: given a non-abelian extension of Lie algebras
$$
0\longrightarrow\mathfrak{h}\overset{i}\longrightarrow\hat{\mathfrak{g}}\overset{p}\longrightarrow\mathfrak{g}\longrightarrow 0,
$$
let $(K,D)\in\mathrm{Der}(\mathfrak{h})\times\mathrm{Der}(\mathfrak{g})$ be a pair of derivations of $\mathfrak{h}$ and $\mathfrak{g}$ respectively. When does there exist a derivation $\hat{D}$ of $\hat{\mathfrak{g}}$ such that $\hat{D}|_{\mathfrak{h}}=K$ and  $D\circ p=p\circ\hat{D}.$ We provide an obstruction class for the existence of such a lift.
\end{abstract}

\subjclass[2020]{17B40, 17B55, 17B56}

\keywords{Non-abelian extension, $(\g, D)$-kernel, Deligne groupoid, Lie $2$-algebra with a derivation, extensible of a derivation.}

\maketitle

\tableofcontents

\allowdisplaybreaks

%\end{document}

\section{Introduction}
Samuel Eilenberg and Saunders MacLane \cite{EM} introduced the notion of non-abelian cohomology of an abstract
group with values in another group to understand non-abelian extensions of abstract groups (see also \cite{HS, Ne1}). Subsequently, non-abelian extensions and the non-abelian cohomology theory were generalized to the context of Lie algebras \cite{Alekseevsky1, Alekseevsky2, Fi, Fr, IK, Ne}. However, unlike classical cohomology theories, neither the non-abelian cohomology of groups nor Lie algebras arises from cochain complexes. Karl-Hermann Neeb \cite{Ne} introduced the notion of a $\g$-kernel(see also \cite{HN, Ho, MM}), to some extent, they used cohomology groups to classify the non-abelian extensions of Lie algebras. This generalizes the classification of abelian extensions of a Lie algebra by a module in terms of the second cohomology space. In \cite{Fr, Nr}, they utilized the differential graded Lie algebra approach to investigate non-abelian extensions of Lie algebras anew. Specifically, they constructed a differential graded Lie algebra and used its Maurer-Cartan elements to character non-abelian extensions of Lie algebras. Moreover, the set of connected components of the Deligne groupoid of the differential graded Lie algebra corresponding to the non-abelian cohomology of Lie algebras. Another classical approach is to construct a Lie 2-algebra from derivations of Lie algebras
and to characterize non-abelian extensions of Lie algebras in terms of Lie 2-algebra homomorphisms \cite{Alekseevsky1, Alekseevsky2, IK, SZ}.

In \cite{TFS}, the authors developed a cohomology theory for Lie algebras with derivations (also termed \LD pairs) and utilized the second cohomology group to classify the central extensions of these structures. Among their results, they also studied the following question:
\begin{itemize}
    \item
Given a central extension of Lie algebras
\[\begin{CD}
0@>>>\h@>i>>\hat{\g}@>p>>\g            @>>>0,
\end{CD}\]
let $(\varphi_\h,\varphi_\g)\in\Der(\h)\times\Der(\g)$ be a pair of derivations of $\h$ and $\g$ respectively. When does there exist a derivation $\varphi_{\hat{\g}}$ of $\hat{\g}$ such that $$\varphi_{\hat{\g}}|_\h=\varphi_\h, \quad \text{and}\quad  \varphi_\g\circ p=p\circ\varphi_{\hat{\g}}.$$ 
\end{itemize}
These results of \LD pairs and the facts of non-abelian extensions of Lie algebras motivate us to generalize their study to the non-abelian extensions of \LD pairs, and to investigate the following question:
\begin{itemize}
    \item
Given a non-abelian extension of Lie algebras
\[\begin{CD}
0@>>>\h@>i>>\hat{\g}@>p>>\g            @>>>0,
\end{CD}\]
let $(K,D)\in\Der(\h)\times\Der(\g)$ be a pair of derivations of $\h$ and $\g$ respectively. When does there exist a derivation $\hat{D}$ of $\hat{\g}$ such that $$\hat{D}|_\h=K, \quad \text{and}\quad  D\circ p=p\circ\hat{D}.$$ 
\end{itemize}

In this paper, we solve the above problem by using non-abelian extensions of \LD pairs. We also investigate non-abelian extensions of \LD pairs through four different approaches. In Section $2$, we introduce the non-abelian cohomology for \LD pairs and apply it to classify non-abelian extensions of \LD pairs. In Section $3$, we construct a differential graded Lie algebra, whose Maurer-Cartan elements are in bijection with the non-abelian $2$-cocycles. Furthermore, the connected components of the Deligne groupoid are used to characterize non-abelian extensions of \LD pairs. In Section $4$, we investigate the relationship between non-abelian extensions of \LD pairs and the homotopy category of strict Lie $2$-algebras with strict derivations. More precisely, we show that the hHom set between strict Lie 2-algebras with strict derivations is isomorphic to the non-abelian cohomology of \LD pairs. We also prove that $(\g,D)\mapsto{\rm{\bf Ext_{nab}}}(\g,D;\h,K)$ is a representable functor in a suitable category. In Section $5$, we introduce the concept of a $(\g, D)$-kernel and prove that two equivalent non-abelian extensions of \LD pairs give rise to the same $(\g, D)$-kernel. We also obtain cohomology groups of \LD pairs via a $(\g, D)$-kernel, then use the second cohomology group to describe the set of equivalent non-abelian extensions of \LD pairs. We use the third cohomology group to study when a $(\g, D)$-kernel comes from a non-abelian extension of \LD pairs. In Section $6$, we end this paper by characterizing the extensibility of derivations of Lie algebras within the established framework. We also give the obstruction class for the extensibility of derivations of Lie algebras. 

Throughout the paper, all vector spaces are over $\mathbb{C}$ and finite-dimensional.

\section{Non-abelian extensions of {\rm LieDer} pairs and the second non-abelian cohomology}
 In this section, we introduce the non-abelian cohomology for {\rm LieDer} pairs and establish a classification of non-abelian extensions in terms of this cohomology.

Let $(\g, [\cdot, \cdot]_\g)$ be a Lie algebra. A derivation of $\g$ is a linear map $D:\g\lon\g$ that satisfies 
$$
D([x, y]_\g)=[D(x), y]_\g+[x, D(y)]_\g, \quad \forall x, y\in\g.
$$
Denote by $\mathrm{Der}(\g)$ the set of derivations of the Lie algebra $(\g, [\cdot, \cdot]_\g)$.
\begin{defi}{\rm(\cite{TFS})}
${\bf (1)}$ A {\bf LieDer} pair is a Lie algebra $(\g, [\cdot, \cdot]_\g)$ with a derivation $D\in\mathrm{Der}(\g)$, 
denoted by $(\g, D)$.

${\bf (2)}$ Let $(\g, D)$ and $(\g',D')$ be two {\rm LieDer} pairs. A {\bf homomorphism} from $(\g,D)$ to $(\g',D')$ is a Lie algebra homomorphism $\psi:\g\rightarrow\g'$ such that
$
D'\circ\psi=\psi\circ D.
$
\end{defi}

\begin{defi}
Let $(\g, D)$ and $(\h, K)$ be two {\rm LieDer} pairs. A {\bf non-abelian $2$-cocycle} of $(\g, D)$ with values in $(\h, K)$ is a triple $(\varrho, \omega, \chi)$ of linear maps $\varrho: \g\lon\Der(\h),~\omega:\wedge^2\g\lon\h$ and $\chi:\g\lon\h$ such that for all $x, y, z\in\g$ and $u, v\in\h$ the following equations hold:
\begin{eqnarray}
\label{nc1}\varrho([x, y]_\g)u&=&\varrho(x)\varrho(y)u-\varrho(y)\varrho(x)u-[\omega(x, y), u]_\h,\\
\label{nc2}\varrho(x)\omega(y, z)+\varrho(y)\omega(z, x)&+&\varrho(z)\omega(x, y)-\omega([x, y]_\g, z)-\omega([y, z]_\g, x)-\omega([z, x]_\g, y)=0,\\
\label{nc3}K(\varrho(x)v)&=&\varrho(D(x))v+\varrho(x)K(v)+[\chi(x), v]_\h,\\
\label{nc4}K(\omega(x, y))+\chi([x, y]_\g)&=&\varrho(x)\chi(y)-\varrho(y)\chi(x)+\omega(D(x), y)+\omega(x, D(y)).
\end{eqnarray}
Denote by $Z^2_{nab}(\g, D;\h, K)$ the set of non-abelian $2$-cocycles.
\end{defi}

\begin{defi}
Two non-abelian $2$-cocycles $(\varrho, \omega, \chi)$ and $(\varrho', \omega', \chi')$ are {\bf equivalent}, if there exists a linear map $\tau:\g\lon\h$ such that for all $x, y\in\g$ and $u\in\h$ the following equations hold:
\begin{eqnarray}
\label{eq11}\varrho(x)u-\varrho'(x)u&=&[\tau(x), u]_\h,\\
\label{eq22}\omega(x, y)-\omega'(x, y)&=&\varrho'(x)\tau(y)-\varrho'(y)\tau(x)+[\tau(x), \tau(y)]_\h-\tau([x, y]_\g),\\
\label{eq33}\chi(x)-\chi'(x)&=&K(\tau(x))-\tau(D(x)).
\end{eqnarray}
The {\bf second non-abelian cohomology} $H^2_{nab}(\g, D;\h, K)$ is defined as the quotient of $Z^2_{nab}(\g, D;\h, K)$ by the above equivalence relation. For any $(\varrho, \omega, \chi)\in Z^2_{nab}(\g, D;\h, K)$, we denote by $[(\varrho, \omega, \chi)]$ its equivalence class of $(\varrho, \omega, \chi)$ in $H^2_{nab}(\g, D; \h, K)$.
\end{defi}

\begin{defi}\label{nab}
Let $(\g,D)$ and $(\h,K)$ be two {\rm LieDer} pairs.

${\bf (1)}$  A {\bf non-abelian extension} of $(\g,D)$ by $(\h,K)$ is a {\rm LieDer} pair $(\hat{\g}, \hat{D})$ such that the following diagram commutes:
\[\begin{CD}
0@>>>\h@>i>>\hat{\g}@>p>>\g            @>>>0\\
@.    @V K VV   @V\hat{D}VV  @V D VV    @.\\
0@>>>\h @>i>>\hat{\g}@>p>>\g             @>>>0.
\end{CD}\]
Denote by ${\rm{\bf ext_{nab}}}(\g, D; \h, K)$ the set of non-abelian extensions of $(\g, D)$ by $(\h, K)$.

${\bf (2)}$ Let $(\hat{\g} ,\hat{D})$ and $(\tilde{\g}, \tilde{D})$ be two non-abelian extensions of the {\rm LieDer} pair $(\g, D)$ by $(\h, K)$. They are said to be {\bf isomorphic} if there exists an isomorphism $\kappa:\tilde{\g}\lon\hat{\g}$ of {\rm LieDer} pairs such that the following diagram commutes:
  \begin{equation*}
\xymatrix@!0{0\ar@{->} [rr]&& \h \ar@{->} [rr] \ar'[d] [dd] \ar@{=} [rd] && \tilde{\g}\ar'[d] [dd]\ar@ {->} [rr] \ar@{->} [rd]^{\kappa}&& \g\ar@{=} [rd]\ar'[d] [dd]\ar@{->} [rr]&&0&\\
&0\ar@{->} [rr]&& \h\ar@{->} [rr]\ar@{->} [dd]&&\hat{\g}\ar@{->} [dd]\ar@{->} [rr]&&\g\ar@ {->} [dd]\ar@{->} [rr]&&0\\
0\ar@{->} [rr]&&\h\ar'[r] [rr] \ar@{=} [rd]&&\tilde{\g}\ar@{->} [rd]^{\kappa}\ar'[r] [rr]&&\g\ar@{=} [rd] \ar'[r] [rr]&&0&\\
&0\ar@{->} [rr]&& \h\ar@{->} [rr]&&\hat{\g}\ar@{->} [rr]&&\g\ar@{->} [rr]&&0.}
\end{equation*}
Denote by ${\rm{\bf Ext_{nab}}}(\g, D; \h, K)$ the set of all equivalence classes of non-abelian extensions of $(\g, D)$ by $(\h, K)$.

${\bf (3)}$ A {\bf section} of an extension $(\hat{\g}, \hat{D})$ of the {\rm LieDer} pair $(\g, D)$ by $(\h, K)$ is a linear map $ s:\g\rightarrow\hat{\g}$ such that
\begin{equation*}
p\circ s=\Id_\g.
\end{equation*}
\end{defi}
\begin{pro}\label{sexnoon}
Let $(\hat{\g}, \hat{D})$ be a non-abelian extension of $(\g, D)$ by $(\h, K)$ and let $s:\g\lon\hat{\g}$ be a section of $(\hat{\g}, \hat{D})$. Define maps $\varrho:\g\rightarrow\Der(\h), ~\omega\in\Hom(\wedge^{2}\g,\h)$ and $\chi\in\Hom(\g,\h)$ by
\begin{eqnarray}\label{111}
\left\{\begin{array}{rcl}
~~\varrho(x)u&=&[s(x),u]_{\hat{\g}}, \\
~~\omega(x,y)&=&[s(x),s(y)]_{\hat{\g}}-s([x,y]_\g), \\
~~\chi(x)&=&\hat{D}(s(x))-s(D(x)), \quad \forall x, y\in\g, u\in\h,
\end{array}\right.
\end{eqnarray}
then $(\varrho, \omega, \chi)\in Z^{2}_{nab}(\g,D;\h, K)$. Moreover, let $s':\g\lon\hat{\g}$ be another section with the corresponding maps $\varrho', \omega'$ and $\chi'$. Then $(\varrho, \omega, \chi)$ and $(\varrho', \omega', \chi')$ are equivalent.
\end{pro}
\begin{proof}
Define $S:\g\oplus \h\lon\hat{\g}$ by
$$
S(x, u)=s(x)+u.
$$
One can use the linear isomorphism $S$ to transfer to $\g\oplus\h$ the {\rm LieDer} pair structure of $\hat{\g}$. In this way we obtain the {\rm LieDer} pair structure $(\g\oplus\h, [\cdot,\cdot]_{\varrho,\omega}, D_{\chi})$, where $[\cdot,\cdot]_{\varrho,\omega}$ and $D_{\chi}$ are given by
\begin{eqnarray*}
[(x, u), (y, v)]_{\varrho,\omega}&=&S^{-1}[S(x, u), S(y, v)]_{\hat{\g}}=\Big([x, y]_\g, \varrho(x)v-\varrho(y)u+[u, v]_\h+\omega(x,y)\Big),\\
D_{\chi}(x, u)&=&S^{-1}\hat{D}S(x, u)=(D(x), K(u)+\chi(x)).
\end{eqnarray*}
By direct calculation, it follows from the fact that $(\g\oplus\h, [\cdot, \cdot]_{\varrho,\omega})$ is a Lie algebra that the following equations hold:
\begin{equation*}\label{eqrep1}
\varrho([x, y]_\g)u=\varrho(x)\varrho(y)u-\varrho(y)\varrho(x)u-[\omega(x, y), u]_\h,\quad \forall x\in\g,~u\in\h,
\end{equation*}
and
\begin{equation*}\label{eqrep2}
\varrho(x)\omega(y, z)+\varrho(y)\omega(z, x)+\varrho(z)\omega(x, y)-\omega([x, y]_\g, z)-\omega([y, z]_\g, x)-\omega([z, x]_\g, y)=0,
\end{equation*}
for all $x, y, z\in\g$.

Moreover, since $D_{\chi}$ is a derivation of $(\g\oplus\h, [\cdot, \cdot]_{\varrho,\omega})$, we obtain that
\begin{equation*}\label{eqrep3}
K(\varrho(x)v)=\varrho(D(x))v+\varrho(x)K(v)+[\chi(x), v]_\h,
\end{equation*}
and
\begin{eqnarray*}\label{eqrep4}
K(\omega(x, y))+\chi([x, y]_\g)
&=&\varrho(x)\chi(y)-\varrho(y)\chi(x)+\omega(D(x), y)+\omega(x, D(y)).
\end{eqnarray*}
From the above four equations, it follows that $(\varrho, \omega, \chi)$ is a non-abelian $2$-cocycle.

Define a linear map $\tau:\g\lon\h$ by
$$
\tau(x)=s(x)-s'(x), \quad \forall x\in\g.
$$
Then $\forall x, y\in\g$ and $\forall u\in\h$ we have that
\begin{eqnarray*}
\varrho(x)u-\varrho'(x)u&=&[s(x), u]_{\hat{\g}}-[s'(x), u]_{\hat{\g}}\\
&=&[\tau(x), u]_{\h},\\
\omega(x, y)-\omega'(x, y)&=&[s(x), s(y)]_{\hat{\g}}-s[x, y]_{\g}-[s'(x), s'(y)]_{\hat{\g}}+s'[x, y]_{\hat{\g}}\\
&=&[\tau(x), s'(y)]_{\hat{\g}}-\tau([x, y]_\g)+[s'(x), \tau(y)]_{\hat{\g}}+[\tau(x), \tau(y)]_\h\\
&=&\varrho'(x)\tau(y)-\varrho'(y)\tau(x)+[\tau(x), \tau(y)]_\h-\tau([x, y]_\g),
\end{eqnarray*}
and
\begin{eqnarray*}
\chi(x)-\chi'(x)&=&\hat{D}(s(x))-s(D(x))-\hat{D}(s'(x))+s'(D(x))\\
&=&\hat{D}(\tau(x))-\tau(D(x)).
\end{eqnarray*}
Thus, non-abelian $2$-cocycles $(\varrho, \omega, \chi)$ and $(\varrho', \omega', \chi')$ are equivalent.
\end{proof}
Let $(\hat{\g}, \hat{D})$ be a non-abelian extension of $(\g, D)$ by $(\h, K)$. Choose a section $s:\g\lon\hat{\g}$. Then $[(\varrho, \omega, \chi)]\in H^2_{nab}(\g,D;\h, K)$ defined by \eqref{111} is independent on the choice of the section according to Proposition \ref{sexnoon}. Thus we can define a map $\Delta:{\rm{\bf ext_{nab}}}(\g, D; \h, K)\lon H^2_{nab}(\g,D;\h,K)$ by
\begin{equation}\label{3333}
\Delta(\hat{\g}, \hat{D})=[(\varrho, \omega, \chi)]\in H^2_{nab}(\g, \h), \quad \forall (\hat{\g}, \hat{D})\in{\rm{\bf ext_{nab}}}(\g, D; \h, K),
\end{equation}
where $\varrho, \omega$ and $\chi$ are given by \eqref{111}.
\begin{pro}\label{pro11}
Let $(\g, D)$ and $(\h, K)$ be {\rm LieDer} pairs. Then the following map is well-defined,
$$
\overline{\Delta}: {\rm{\bf Ext_{nab}}}(\g, D; \h, K)\lon H^2_{nab}(\g,D;\h, K),\quad \overline{\Delta}([(\hat{\g}, \hat{D})])=\Delta(\hat{\g}, \hat{D}), \quad \forall~~ [(\hat{\g}, \hat{D})]\in{\rm{\bf Ext_{nab}}}(\g, D; \h, K).
$$
\end{pro}
\begin{proof}
Let $(\hat{\g}, \hat{D})$ and $(\tilde{\g}, \tilde{D})$ be two isomorphic non-abelian extensions of $(\g, D)$ by $(\h, K)$. Then there exists an isomorphism $\kappa:\tilde{\g}\lon\hat{\g}$ of {\rm LieDer} pairs such that the following diagram commutes:
  \begin{equation*}
\xymatrix@!0{0\ar@{->} [rr]&& \h \ar@{->} [rr] \ar'[d] [dd] \ar@{=} [rd] && \tilde{\g}\ar'[d] [dd]\ar@ {->} [rr] \ar@{->} [rd]^{\kappa}&& \g\ar@{=} [rd]\ar'[d] [dd]\ar@{->} [rr]&&0&\\
&0\ar@{->} [rr]&& \h\ar@{->} [rr]\ar@{->} [dd]&&\hat{\g}\ar@{->} [dd]\ar@{->} [rr]&&\g\ar@ {->} [dd]\ar@{->} [rr]&&0\\
0\ar@{->} [rr]&&\h\ar'[r] [rr] \ar@{=} [rd]&&\tilde{\g}\ar@{->} [rd]^{\kappa}\ar'[r] [rr]&&\g\ar@{=} [rd] \ar'[r] [rr]&&0&\\
&0\ar@{->} [rr]&& \h\ar@{->} [rr]&&\hat{\g}\ar@{->} [rr]&&\g\ar@{->} [rr]&&0.}
\end{equation*}
Assume that $s:\g\lon\tilde{\g}$ is a section, and $(\tilde{\varrho}, \tilde{\omega}, \tilde{\chi})$ is the corresponding non-abelian $2$-cocycle. Define $s':\g\lon\hat{\g}$ by
 $s'=\kappa\circ s$, then $s'$ is also a section. We denote by $(\hat{\varrho}, \hat{\omega}, \hat{\chi})$ the corresponding non-abelian $2$-cocycle. Since $\kappa|_{\h}=\Id_{\h}$, we have that
\begin{eqnarray*}
\tilde{\varrho}(x)u-\hat{\varrho}(x)u&=&[s(x), u]_{\tilde{\g}}-[\kappa(s(x)),u]_{\hat{\g}}=[s(x), u]_{\tilde{\g}}-\kappa([s(x),u]_{\tilde{\g}})\\
&=&0,\\
\tilde{\omega}(x, y)-\hat{\omega}(x, y)%&=&[s(x), s(y)]_{\tilde{\g}}-s[x, y]_{\g}-[\kappa(s(x)), \kappa(s(y))]_{\hat{\g}}+\kappa(s[x, y]_{\g})\\
&=&[s(x), s(y)]_{\tilde{\g}}-s[x, y]_{\g}-\kappa([s(x), s(y)]_{\hat{\g}}-s[x, y]_{\g})\\
&=&0,
\end{eqnarray*}
and
\begin{eqnarray*}
\tilde{\chi}(x)-\hat{\chi}(x)%&=&\tilde{D}(s(x))-s(D(x))-\hat{D}(\kappa(s(x)))+\kappa(s(D(x)))\\
&=&\tilde{D}(s(x))-s(D(x)-\kappa(\tilde{D}(s(x))-s(D(x)))\\
&=&0,
\end{eqnarray*}
which means that $(\tilde{\varrho}, \tilde{\omega}, \tilde{\chi})$ and $(\hat{\varrho}, \hat{\omega}, \hat{\chi})$ are equivalent. Thus the map
$$
\overline{\Delta}: {\rm{\bf Ext_{nab}}}(\g, D; \h, K)\lon H^2_{nab}(\g,D;\h, K),\quad \overline{\Delta}([(\hat{\g}, \hat{D})])=\Delta(\hat{\g}, \hat{D}), \quad \forall~~ [(\hat{\g}, \hat{D})]\in{\rm{\bf Ext_{nab}}}(\g, D; \h, K),
$$
is well-defined.
\end{proof}
\begin{thm}\label{corres}
Let $(\g, D)$ and $(\h, K)$ be two {\rm LieDer} pairs. Then the non-abelian extensions of $(\g, D)$ by $(\h, K)$ are classified by the second non-abelian cohomology $H_{nab}^2
(\g, D; \h, K)$.
\end{thm}
\begin{proof}
By Proposition \ref{pro11}, there is a map
$$
\overline{\Delta}: {\rm{\bf Ext_{nab}}}(\g, D; \h, K)\lon H^2_{nab}(\g, D;\h, K),\quad \overline{\Delta}([(\hat{\g}, \hat{D})])=\Delta(\hat{\g}, \hat{D}), \quad \forall~~ [(\hat{\g}, \hat{D})]\in{\rm{\bf Ext_{nab}}}(\g, D; \h, K).
$$
%We only need to prove that $\overline{\Delta}$ is bijective.

Suppose that $(\varrho, \omega, \chi)\in Z^2_{nab}(\g,D;\h, K)$, define a bracket on $\g\oplus\h$ by
$$
[(x, u), (y, v)]_{\varrho, \omega}=([x, y]_\g, \varrho(x)v-\varrho(y)u+[u, v]_\h+\omega(x, y)), \quad\forall x, y\in\g, u, v\in\h.
$$
By \eqref{nc1} and \eqref{nc2}, we have that $(\g\oplus\h, [\cdot, \cdot]_{\varrho, \omega})$ is a Lie algebra. Define a linear map $D_\chi: \g\oplus\h\lon\g\oplus\h$ by
$$
D_\chi(x, u)=(D(x), K(u)+\chi(x)), \quad \forall x\in\g, \in\h.
$$
By \eqref{nc3} and \eqref{nc4}, it is straightforward to deduce that $D_\chi$ is a derivation on $(\g\oplus\h, [\cdot, \cdot]_{\varrho, \omega})$. If $(\varrho', \omega', \chi')$ and $(\varrho, \omega, \chi)$ are two equivalent non-abelian $2$-cocycles, then there is a linear map $\tau:\g\lon\h$ satisfying
\begin{eqnarray}\label{09}
\left\{\begin{array}{rcl}
\varrho(x)u-\varrho'(x)u&=&[\tau(x), u]_\h,\\
\omega(x, y)-\omega'(x, y)&=&\varrho'(x)\tau(y)-\varrho'(y)\tau(x)+[\tau(x), \tau(y)]_\h-\tau([x, y]_\g),\\
\chi(x)-\chi'(x)&=&K(\tau(x))-\tau(D(x))
\end{array}\right.
\end{eqnarray}
Denote by $(\g\oplus\h, [\cdot, \cdot]_{\varrho', \omega'}, D_{\chi'})$ the {\rm LieDer} pair corresponding to $(\varrho', \omega', \chi')$. Define a linear map $\kappa:\g\oplus\h\lon\g\oplus\h$ by
$$
\kappa(x, u)=(x, u+\tau(x)), \quad\forall x\in\g, u\in\h.
$$
By \eqref{09}, it is straightforward to deduce that $\kappa:(\g\oplus\h, [\cdot, \cdot]_{\varrho, \omega}, D_{\chi})\lon (\g\oplus\h, [\cdot, \cdot]_{\varrho', \omega'}, D_{\chi'})$ is an isomorphism of non-abelian extensions of $(\g, D)$ by $(\h, K)$. Thus, we have a map $\Theta: H_{nab}^2(\g,D; \h, K)\lon {\rm{\bf Ext_{nab}}}(\g, D; \h, K)$ defined by
$$
\Theta([(\varrho, \omega, \chi)])=[(\g\oplus\h, [\cdot, \cdot]_{\varrho, \omega}, D_\chi)], \quad  \forall [(\varrho, \omega, \chi)]\in H_{nab}^2(\g, D; \h, K),
$$
where $[(\g\oplus\h, [\cdot, \cdot]_{\varrho, \omega}, D_\chi)]$ is the equivalence class of $(\g\oplus\h, [\cdot, \cdot]_{\varrho, \omega}, D_\chi)$. Moreover, we have $$\Theta\circ\bar{\Delta}=\Id, \quad \bar{\Delta}\circ\Theta=\Id,$$ which means that the non-abelian extensions of $(\g, D)$ by $(\h, K)$ are classified by the second non-abelian cohomology $H_{nab}^2
(\g, D; \h, K)$.
\end{proof}

\section{Non-abelian extensions of \rm{LieDer} pairs and Deligne groupoids}
In this section,  we show that the theory of non-abelian extensions of \LD pairs can be understood in terms of
differential graded Lie algebras. More precisely, the equivalence classes of non-abelian extensions of a \LD pair $(\g, D)$ by a \LD pair $(\h, K)$ can be seen as the $\pi_0$ of the {\rm Deligne} groupoid of a differential graded Lie algebra $(\huaL^{\g, \h}=\oplus_{n\in\mathbb{Z}}\huaL^{\g, \h}_n, [\cdot, \cdot]_\huaL, d)$.
We first recall some definitions related to differential graded Lie algebras.

\begin{defi}{\rm(\cite{NR})}
${\bf (1)}$ A {\bf graded Lie algebra} is a $\mathbb{Z}$-graded vector space $\huaL=\oplus_{n\in\mathbb{Z}}\huaL_{n}$ equipped with a degree-preserving bilinear bracket $[\cdot,\cdot]_{\huaL}:\huaL\times\huaL\lon\huaL$ which satisfies
\begin{eqnarray*}
    [a,b]_{\huaL}&=&-(-1)^{|a||b|}[b,a]_{\huaL},\\
    ~~[a,[b,c]_{\huaL}]_{\huaL}&=&[[a,b]_{\huaL},c]_{\huaL}+(-1)^{|a||b|}[b,[a,c]_{\huaL}]_{\huaL}
\end{eqnarray*}
for all homogeneous elements $a,b,c\in \huaL$, with $|a|$ and $|b|$ as degree of $a$ and $b$ respectively. 

${\bf (2)}$ A {\bf differential graded Lie algebra} is a graded Lie algebra $(\huaL,[\cdot,\cdot]_{\huaL})$ equipped with a homological derivation $d:\huaL\lon\huaL$ of degree $1$, i.e.,
    \begin{eqnarray*}
        |d(a)|&=&|a|+1,\\
        d([a,b]_{\huaL})&=&[d(a),b]_{\huaL}+(-1)^{|a|}[a,d(b)]_{\huaL},\\
        d^{2}&=&0,
    \end{eqnarray*}
    for all homogeneous elements $a,b\in\huaL$.
\end{defi}
\begin{defi}{\rm(\cite{NR})}
A {\bf Maurer-Cartan element} of a differential graded Lie algebra $(\huaL,[\cdot,\cdot]_{\huaL},d)$ is an element $\mu\in \huaL_{1}$ satisfying the {\bf Maurer-Cartan equation}
\begin{equation}
    d\mu+\frac{1}{2}[\mu,\mu]_{\huaL}=0.
\end{equation}
Denote by $\mathrm{MC}(\huaL)$ the set of {\rm Maurer-Cartan} elements of $(\huaL,[\cdot,\cdot]_{\huaL},d)$.
\end{defi}
\begin{pro}\label{twistdg}{\rm (\cite{DSV, Get})}
If $\alpha$ is a Maurer-Cartan element of $(\huaL,[\cdot,\cdot]_{\huaL},d)$, then $(\huaL,[\cdot,\cdot]_{\huaL}, d+[\alpha, \cdot]_{\huaL})$ is a differential graded Lie algebra, called the twisted differential graded Lie algebra.
\end{pro}
There is an equivalence relation $\sim$ on $\mathrm{MC}(\huaL)$ known as gauge equivalence. More precisely, two {\rm Maurer-Cartan} elements  $\mu$ and $\mu'$ of $(\huaL,[\cdot,\cdot]_{\huaL},d)$ are called {\bf equivalent}, if there exists an element $\tau\in \huaL_{0}$ and $\ad\tau$ is nilpotent, such that 
    \begin{equation*}
        \mu=e^{\ad\tau}\mu'+\frac{\Id-e^{\ad\tau}}{\ad\tau}(d\tau).
    \end{equation*}
The element $\tau$ is called to be a {\bf gauge transformation}.

\begin{defi}{\rm(\cite{Fr})}
    Let $(\huaL,[\cdot,\cdot]_{\huaL},d)$ be a differential graded Lie algebra with $\huaL_0$ abelian. The {\bf Deligne groupoid} $\mathrm{Del}(\huaL)$ associated to $(\huaL,[\cdot,\cdot]_{\huaL},d)$ is defined as follows: its set of objects is ${\rm MC}(\huaL)$, and for any $x, y\in {\rm MC}(\huaL)$, the set of morphisms $\Hom(x, y)$ is empty if $x$ and $y$ are not gauge equivalent, and is given by the gauge transformations between $x$ and $y$ otherwise. One defines its set of connected components $\pi_0$:
    $$
    \pi_0(\mathrm{Del}(\huaL)):={\rm MC}/\sim.
    $$
\end{defi}
Let $\g$ be a vector space. Define the graded vector space
$\oplus_{n=0}^\infty \Hom(\wedge^{n+1}\g,\g)$
with the degree of elements in $\Hom(\wedge^n\g,\g)$ being $n-1$. For $\alpha\in \Hom(\wedge^m\g,\g), \beta\in \Hom(\wedge^n\g,\g)$, the  Nijenhuis-Richardson bracket $[\cdot,\cdot]_\NR$ is defined by
$$ [\alpha,\beta]_{\NR}:=\alpha\circ_\NR\beta- (-1)^{(m-1)(n-1)}\beta\circ_{\NR}\alpha,$$
where $\alpha\circ_{\NR}\beta\in \Hom(\wedge^{m+n-1}\g,\g)$ is given by
\begin{equation}
(\alpha\circ_\NR\beta)(x_1,\cdots,x_{m+n-1}):=\sum_{\sigma\in S(n,m-1)} (-1)^\sigma \alpha(\beta(x_{\sigma(1)},\cdots,x_{\sigma(n)}),x_{\sigma(n+1)}, \cdots,x_{\sigma(m+n-1)}),
\label{eq:fgcirc}
\end{equation}
with the sum taken over all $(n,m-1)$-shuffles.
Then $\big(\oplus_{n=0}^\infty \Hom(\wedge^{n+1}\g,\g),[\cdot,\cdot]_{\NR}\big)$ is a graded Lie algebra \cite{NR,NR2}. With this setup, a Lie algebra structure on $\g$ is precisely a degree 1 solution $\pi\in \Hom(\wedge^2\g,\g)$ of the {\rm Maurer-Cartan} equation
$$ [\pi,\pi ]_{\NR}=0.$$

Let $\g$ be a vector space. Define a graded vector space $\oplus_{n=1}^{+\infty}V_n$ by
$$
V_n=\Hom(\wedge^{n+1}\g, \g)\oplus\Hom(\wedge^n\g, \g).
$$
Define a bilinear map $[\cdot, \cdot]_{\LD}:\oplus_{n=1}^{+\infty}V_n\times\oplus_{n=1}^{+\infty}V_n\lon\oplus_{n=1}^{+\infty}V_{n}$ by
$$
[(\alpha, P), (\beta, Q)]_\LD=([\alpha, \beta]_\NR, [\alpha, Q]_\NR-(-1)^{kl}[\beta, P]_{\NR}),
$$
where $(\alpha, P)\in V_{k}, (\beta, Q)\in V_{l}$. Then $(\oplus_{n=1}^{+\infty}V_n, [\cdot, \cdot]_\LD)$ is a graded Lie algebra\cite{JS,LQYZ,TFS}. Furthermore, a \LD pair on $\g$ is precisely a degree $1$ solution $(\pi, D)\in\Hom(\wedge^2\g, \g)\oplus\Hom(\g, \g)$ of the {\rm Maurer-Cartan} equation
$$
[(\pi, D), (\pi, D)]_\LD=0.
$$

Let $\g$ and $\h$ be vector spaces. The elements in $\g$ are denoted by $x_i$ and the elements in $\h$ are denoted by $u_j$. For a multilinear map $\kappa: \wedge^{k}\g\otimes\wedge^{l}\h\lon \g,$ we define $\hat{\kappa}\in\Hom(\wedge^{k+l}(\g\oplus \h), \g\oplus \h)$ by
\begin{equation*}
\hat{\kappa}(x_1+u_1,\cdots,x_{k+l}+u_{k+l})=\sum_{\sigma\in S(k,l)}(-1)^{\sigma}\Big(\kappa(x_{\sigma(1)},\cdots,x_{\sigma(k)},u_{\sigma(k+1)},\cdots,u_{\sigma(k+l)}),0\Big).
\end{equation*}
Similarly, for $\kappa: \wedge^{k}\g\otimes\wedge^{l}\h\lon \h,$ we define $\hat{\kappa}\in\Hom(\wedge^{k+l}(\g\oplus \h), \g\oplus \h)$ by
\begin{equation*}
\hat{\kappa}(x_1+u_1,\cdots,x_{k+l}+u_{k+l})=\sum_{\sigma\in S(k,l)}(-1)^{\sigma}\Big(0, \kappa(x_{\sigma(1)},\cdots,x_{\sigma(k)},u_{\sigma(k+1)},\cdots,u_{\sigma(k+l)})\Big).
\end{equation*}
The linear map $\hat{\kappa}$ is called a {\bf lift} of $\kappa$.
Denote by $$(\g, \h)^{k,l}=\wedge^{k}\g\otimes\wedge^{l}\h,$$ then $$\wedge^{n}(\g\oplus \h)\cong\oplus_{k+l=n}(\g, \h)^{k,l},$$
and $$\Hom(\wedge^{n}(\g\oplus \h), \g\oplus \h)\cong(\oplus_{k+l=n}\Hom((\g, \h)^{k,l},\g))\oplus(\oplus_{k+l=n}\Hom((\g, \h)^{k,l},\h)),$$ where the isomorphism is the lift. The lift can be used to transfer to $$\oplus^{+\infty}_{n=1}\Big((\oplus_{k+l=n}\Hom((\g, \h)^{k,l},\g))\oplus(\oplus_{k+l=n}\Hom((\g, \h)^{k,l},\h))\Big)$$ the graded Lie algebra structure of $(\oplus_{n=1}^{+\infty}\Hom(\wedge^{n}(\g\oplus\h), (\g\oplus\h)), [\cdot, \cdot]_{\NR})$. 
\begin{lem}\label{lem1111}
With the above notations,
$$\Big(\oplus^{+\infty}_{n=1}\Big((\oplus_{k+l=n}\Hom((\g, \h)^{k,l},\g))\oplus(\oplus_{k+l=n}\Hom((\g, \h)^{k,l},\h))\Big), ~~[\cdot, \cdot]\Big),$$ 
is a graded Lie algebra, where $[\cdot,\cdot]$ satisfies
\begin{equation}\label{defibra}
\widehat{[f, g]}=[\hat{f}, \hat{g}]_{\NR}.
\end{equation}
\end{lem}
In particular, by \eqref{defibra}, we have that
\begin{eqnarray}\label{11111}
\left\{\begin{array}{rcl}
~~[f, g] &=&[f, g]_{\NR},\quad f\in\Hom(\wedge^i\g, \g), g\in\Hom(\wedge^{j}\g, \g), \\
~~[f, g]&=&[f, g]_{\NR}, \quad f\in\Hom(\wedge^i\h, \h), g\in\Hom(\wedge^j\h, \h),  \\
~~[f, g]&=&0, \quad\quad\quad\quad f\in\Hom(\wedge^{i}\g, \h), g\in\Hom(\wedge^{j}\g, \h)\\
~~[f_i, g_j]&\in&\Hom((\g, \h)^{i+j-1, l}, \h), \quad f_i\in\Hom(\wedge^{i}\g, \g), g_j\in\Hom((\g, \h)^{j, l}, \h)\\
~~[f_l, g_i]&\in&\Hom((\g, \h)^{l+i-1, j}, \h), \quad f_l\in\Hom(\wedge^{l}\g, \h), g_i\in\Hom((\g, \h)^{i, j}, \h)\\
~~[f_i, g_t]&\in&\Hom((\g, \h)^{i+t, j+s-1}, \h), \quad f_i\in\Hom((\g, \h)^{i, j}, \h), ~~g_t\in\Hom((\g, \h)^{t, s}, \h).
\end{array}\right.
\end{eqnarray}

Let $\g$ and $\h$ be vector spaces. Define a graded vector space $\huaL^{\g, \h}=\oplus_{n\in\mathbb{Z}}\huaL^{\g, \h}_n$ as follows
\begin{equation*}
\huaL^{\g, \h}_n=
\begin{cases}
0, & n\leq-1;\\
\Hom(\g, \h)\oplus 0, & n=0;\\
\Big(\oplus_{\substack{i+j=n+1 \\ i\geq 1}}\Hom(\wedge^{i}\g\otimes\wedge^{j}\h, \h)\Big)\oplus\Big(\oplus_{\substack{i+j=n \\ i\geq 1}}\Hom(\wedge^{i}\g\otimes\wedge^{j}\h, \h)\Big), & n\geq 1.
\end{cases}
\end{equation*}
Let $\pi_\g\in\Hom(\wedge^{2}\g, \g)$ and $\pi_\h\in\Hom(\wedge^{2}\h, \h)$, $D\in\Hom(\g, \g)$ and $K\in\Hom(\h, \h)$ satisfy
$$
[\pi_\g, \pi_\g]=0,\quad [\pi_\g, D]=0, \quad [\pi_\h, \pi_\h]=0, \quad [\pi_\h, K]=0.
$$
By \eqref{defibra} and \eqref{11111}, we have that $(\g, \pi_\g)$ and $(\h, \pi_\h)$ are Lie algebras,  $D:\g\lon\g$ and $K:\h\lon\h$ are derivations on $(\g, \pi_\g)$ and $(\h, \pi_\h)$ respectively. Thus $(\pi_\g+\pi_\h, D+K)$ is a Maurer-Cartan element of $\Big(\oplus^{+\infty}_{n=1}\Big((\oplus_{k+l=n}\Hom((\g, \h)^{k,l},\g))\oplus(\oplus_{k+l=n}\Hom((\g, \h)^{k,l},\h))\Big), ~~[\cdot, \cdot]\Big)$. By Proposition \ref{twistdg}, we have the following result.

\begin{pro}
With the above notations, $(\huaL^{\g, \h}=\oplus_{n\in\mathbb{Z}}\huaL^{\g, \h}_n, [\cdot, \cdot]_{\huaL^{\g, \h}}, d)$ is a differential graded Lie algebra, and $\huaL^{\g, \h}_0$ is abelian, where $[\cdot, \cdot]_{\huaL^{\g, \h}}: \huaL^{\g, \h}\times\huaL^{\g, \h}\lon\huaL^{\g, \h}$ and $d:\huaL^{\g, \h}\lon\huaL^{\g, \h}$ are defined by
\begin{eqnarray*}
[(f, \alpha), (g, \beta)]_{\huaL^{\g, \h}}&=&([f, g], [f, \beta]-(-1)^{kl}[g, \alpha]),\\
d(f, \alpha)&=&([\pi_\g+\pi_\h, f], [\pi_\g+\pi_\h, \alpha]-(-1)^{k}[f, D+K]),
\end{eqnarray*}
for all $f\in\oplus_{\substack{i+j=k+1\\i\geq 1}}\Hom(\wedge^{i}\g\otimes\wedge^{j}\h, \h), g\in\oplus_{\substack{i+j=l+1\\i\geq 1}}\Hom(\wedge^{i}\g\otimes\wedge^{j}\h, \h), \alpha\in\oplus_{\substack{i+j=k\\i\geq 1}}\Hom(\wedge^{i}\g\otimes\wedge^{j}\h, \h)$ and $\beta\in\oplus_{\substack{i+j=l\\i\geq 1}}\Hom(\wedge^{i}\g\otimes\wedge^{j}\h, \h)$.
\end{pro}

\begin{thm}\label{zmc}
The set of {\rm Maurer-Cartan} elements of $(\huaL^{\g, \h}=\oplus_{n\in\mathbb{Z}}\huaL^{\g, \h}_n, [\cdot, \cdot]_{\huaL^{\g, \h}}, d)$ is $Z^{2}_{nab}(\g,D;\h,K)$, i.e.,
    \begin{equation}
        Z^{2}_{nab}(\g,D;\h,K)\cong\mathrm{MC}(\huaL^{\g, \h}),
    \end{equation}
where $Z^{2}_{nab}(\g,D; \h,K)$ is the set of non-abelian 2-cocycles. 
    \begin{proof}
        Suppose $\mu=(\omega+\varrho,\chi)\in \mathrm{MC}(\huaL^{\g, \h})$, where $\omega\in\Hom(\wedge^{2}\g,\h),\ \varrho\in\Hom(\g\otimes\h,\h)$ and $\chi\in\Hom(\g,\h)$. Then by the fact that $\mu=(\omega+\varrho,\chi)$ satisfies the {\rm Maurer-Cartan} equation, i.e.,
        \begin{equation*}
            d(\omega+\varrho,\chi)+\frac{1}{2}[(\omega+\varrho,\chi),(\omega+\varrho,\chi)]_{\huaL^{\g, \h}}=0,
        \end{equation*}
        we obtain that 
        \begin{eqnarray*}
            0&=&d(\omega+\varrho,\chi)+\frac{1}{2}[(\omega+\varrho,\chi),(\omega+\varrho,\chi)]_{\huaL^{\g, \h}}\\
            %&=&([\pi_{\g}+\pi_{\h},\omega+\varrho],[\pi_{\g}+\pi_{\h},\chi]+[\omega+\varrho,D+K])\\
            %&&+(\frac{1}{2}[\omega+\varrho,\omega+\varrho],[\omega+\varrho,\chi])\\
            &=&([\pi_{\g}+\pi_{\h},\omega+\varrho]+\frac{1}{2}[\omega+\varrho,\omega+\varrho],[\pi_{\g}+\pi_{\h},\chi]+[\omega+\varrho,D+K]+[\omega+\varrho,\chi]).
        \end{eqnarray*}
        In other words, we get 
        \begin{equation*}
            [\pi_{\g}+\pi_{\h},\omega+\varrho]+\frac{1}{2}[\omega+\varrho,\omega+\varrho]=0\in\Hom(\wedge^{3}\g,\h)\oplus\Hom(\wedge^{2}\g\otimes\h,\h)\oplus\Hom(\g\otimes\wedge^{2}\h,\h),
        \end{equation*}
        and 
        \begin{equation*}
            [\pi_{\g}+\pi_{\h},\chi]+[\omega+\varrho,D+K]+[\omega+\varrho,\chi]=0\in\Hom(\wedge^{2}\g,\h)\oplus\Hom(\g\otimes\h,\h).
        \end{equation*}
    By \eqref{11111}, we have 
        \begin{eqnarray*}
            \Hom(\wedge^{3}\g,\h)&:&[\pi_{\g},\omega]+[\omega,\varrho]=0,\\
            \Hom(\wedge^{2}\g\otimes\h,\h)&:& [\pi_{\g},\varrho]+[\omega.\pi_{\h}]+\frac{1}{2}[\varrho,\varrho]=0,\\
            \Hom(\g\otimes\wedge^{2}\h,\h)&:&[\varrho,\pi_{\h}]=0,
        \end{eqnarray*}
        as well as
        \begin{eqnarray*}
            \Hom(\wedge^{2}\g,\h)&:&[\pi_{\g},\chi]+[\omega,D]+[\omega,K]+[\varrho,\chi]=0,\\
            \Hom(\g\otimes\h,\h)&:&[\varrho,D]+[\varrho,K]+[\pi_{\h},\chi]=0.
        \end{eqnarray*}
        The first, second, fourth and fifth equations correspondence to (\ref{nc2}) and (\ref{nc1}), (\ref{nc4}) and (\ref{nc3}) respectively. Furthermore, the third equation implies that $\varrho(x,\cdot)\in\Der(\h)$ for all $x\in\g$. Therefore, the triple $(\varrho,\omega,\chi)$ is a non-abelian 2-cocycle.

        Conversely, let $(\varrho,\omega,\chi)\in Z^{2}_{nab}(\g,D;\h,K)$ be a non-abelian 2-cocycle of $(\g,D)$ with values in $(\h,K)$. Define $\mu=(\omega+\varrho,\chi)\in \huaL^{\g, \h}_{1}$, by direct calculation, one obtains that $\mu$ is a Maurer-Cartan element of the differential graded Lie algebra $(\huaL^{\g, \h},[\cdot,\cdot]_{\huaL^{\g, \h}},d)$.
    \end{proof}
\end{thm}    

\begin{thm}
    The set of connected components of $\mathrm{Del}(\huaL^{\g, \h})$ is $H^2_{nab}(\g,D;\h, K)$, i.e.,
    $$
    \pi_0(\mathrm{Del}(\huaL^{\g, \h}))\cong H^2_{nab}(\g, D;\h, K)\cong\mathbf{Ext_{nab}}(\g,D;\h,K).
    $$
    \begin{proof}
       By Theorem \ref{zmc}, the set of objects of  $\mathrm{Del}(\huaL^{\g, \h})$ is isomorphic to $Z^{2}_{nab}(\g,D; \h,K)$. Then, to prove $$\pi_0(\mathrm{Del}(\huaL^{\g, \h}))\cong H^2_{nab}(\g,D; \h,K)\cong\mathbf{Ext_{nab}}(\g,D;\h,K),$$ it suffices to show that $\Hom_{\mathrm{Del}(\huaL^{\g, \h})}(\mu, \mu')$ is non-empty iff $[\mu]=[\mu']$, where
       $\mu=(\varrho, \omega, \chi), \mu'=(\varrho', \omega', \chi')\in Z^{2}_{nab}(\g,D; \h,K)$. 

        If $\Hom_{\mathrm{Del}(\huaL^{\g, \h})}(\mu, \mu')\neq \emptyset$, then there exists $(\tau,0)\in \huaL^{\g, \h}_{0}$ such that,
        \begin{equation}\label{mceq}
            \mu=e^{\ad(\tau,0)}\mu'+\frac{\Id-e^{\ad(\tau,0)}}{\ad(\tau,0)}d(\tau,0).
        \end{equation}
        By \eqref{11111}, it follows that $[\tau,\omega']=0,\ [\tau,\chi']=0$ and $[\tau,[\tau,\varrho']]=0$. Therefore, 
        \begin{equation*}
            \ad(\tau,0)(\mu')=[(\tau,0),(\omega'+\varrho',\chi')]=([\tau,\omega']+[\tau,\varrho'],[\tau,\chi'])=([\tau,\varrho'],0),
        \end{equation*}
        and consequently, $$\ad^{2}(\tau,0)(\mu')=[(\tau,0),([\tau,\varrho'],0)]=([\tau,[\tau,\varrho']],0)=0.$$
        Hence, we have
        \begin{eqnarray*}
            e^{\ad(\tau,0)}\mu'&=&\sum_{n\geq0}\frac{1}{n!}\ad^{n}(\tau,0)(\mu')%=\big(\Id+\ad(\tau,0)+\frac{1}{2}\ad^{2}(\tau,0)+\cdots\big)(\mu')\\
            =\big(\Id+\ad(\tau,0)\big)(\mu')\\%=(\omega'+\varrho',\chi')+([\tau,\varrho'],0)\\
            &=&(\omega'+\varrho'+[\tau,\varrho'],\chi').
        \end{eqnarray*}
        By \eqref{11111}, we also have $$[\tau,[\pi_{\g},\tau]]=0, \quad [\tau, [\tau, [\pi_\h, \tau]]]=0,\quad [\tau,[\tau,D]]=0, \quad [\tau,[\tau,K]]=0.$$ Therefore,
        \begin{eqnarray*}
            [(\tau,0),d(\tau,0)]&=&[(\tau,0),([\pi_{\g},\tau]+[\pi_{\h},\tau],-[\tau,D]-[\tau,K])]\\
           % &=&([\tau,[\pi_{\g},\tau]]+[\tau,[\pi_{\h},\tau]],-[\tau,[\tau,D]]-[\tau,[\tau,K]])\\
            &=&([\tau,[\pi_{\h},\tau]],0),
        \end{eqnarray*}
        and 
        $$\ad^{2}(\tau,0)(d(\tau,0))=[(\tau,0),[(\tau,0),d(\tau,0)]]=[(\tau,0),([\tau,[\pi_{\h},\tau]],0)]=([\tau,[\tau,[\pi_{\h},\tau]]],0)=0.$$
        Consequently,
        \begin{eqnarray*}
            \frac{\Id-e^{\ad(\tau,0)}}{\ad(\tau,0)}d(\tau,0)&=&-\sum_{n\geq0}\frac{1}{(n+1)!}\ad^{n}(\tau,0)d(\tau,0)\\
            %&=&-\Big(d(\tau,0)+\frac{1}{2}[(\tau,0),d(\tau,0)]+\frac{1}{6}[(\tau,0),[(\tau,0),d(\tau,0)]]+\cdots\Big)\\
           % &=&-\Big(([\pi_{\g},\tau]+[\pi_{\h},\tau],-[\tau,D]-[\tau,K])+\frac{1}{2}([\tau,[\pi_{\h},\tau]],0)\Big)\\
            &=&-([\pi_{\g},\tau]+[\pi_{\h},\tau]+\frac{1}{2}[\tau,[\pi_{\h},\tau]],-[\tau,D]-[\tau,K]).
        \end{eqnarray*}
        We can rewrite equation (\ref{mceq}) and regroup the result according to the corresponding components, we obtain the following equations
        \begin{eqnarray*}
            \Hom(\wedge^{2}\g,\h)&:& \omega=\omega'+[\tau,\varrho']-[\pi_{\g},\tau]-\frac{1}{2}[\tau,[\pi_{\h},\tau]],\\
            \Hom(\g\otimes\h,\h)&:& \varrho=\varrho'-[\pi_{\h},\tau],\\
            \Hom(\g,\h)&:&   \chi=\chi'+[\tau,D]+[\tau,K].
        \end{eqnarray*}
        The above three equations give us the equalities (\ref{eq22}) (\ref{eq11}) and (\ref{eq33}) respectively. Hence, two non-abelian 2-cocycles $(\varrho,\omega,\chi)$ and $(\varrho',\omega',\chi')$ are equivalent through $\tau\in\Hom(\g,\h)$.

    If $(\varrho,\omega,\chi)$ and $(\varrho',\omega',\chi')$ are two equivalent non-abelian 2-cocycles, i.e., there exists a linear map $\tau:\g\lon\h$, such that the equations (\ref{eq11}), (\ref{eq22}) and (\ref{eq33}) are satisfied. Then by direct calculation, one can see that $\mu:=(\omega+\varrho,\chi)$ and $\mu':=(\omega'+\varrho',\chi')$ are equivalent through $e^{\ad(\tau,0)}$, which means that $\Hom_{\mathrm{Del}(\huaL^{\g, \h})}(\mu, \mu')\neq \emptyset$.
Therefore, $\Hom_{\mathrm{Del}(\huaL^{\g, \h})}(\mu, \mu')$ is non-empty iff $[\mu]=[\mu']$, which implies that
$$
\pi_0(\mathrm{Del}(\huaL^{\g, \h}))\cong H^2_{nab}(\g, D; \h,K)\cong\mathbf{Ext_{nab}}(\g,D;\h,K).$$
We finish the proof.        
\end{proof}
\end{thm}

\section{Non-abelian extensions of \LD pairs and the homotopy category of strict Lie 2-algebras with strict derivations}
In this section, we investigate the relationship between non-abelian extensions of \LD pairs and hHom sets between strict Lie 2-algebras with strict derivations. We show that $(\g,D)\mapsto{\rm{\bf Ext_{nab}}}(\g,D;\h,K)$ is a representable functor in a suitable category. We recall some definitions of strict Lie $2$-algebras firstly. 
\begin{defi}{\rm(\cite{Baez})}
A {\bf strict Lie 2-algebra} is a graded vector space $\g=\g_{0}\oplus\g_{1}$ with the following maps 
\begin{itemize}
  \item a linear map $d_\g:\g_{1}\lon \g_{0}$,
  \item skew-symmetric bilinear maps $l_{2}:\g_{0}\wedge \g_{0}\lon \g_{0}$ and $\g_{0}\wedge \g_{1}\lon \g_{1}$, also denoted by $[\cdot,\cdot]$,
\end{itemize}
which satisfy
\begin{equation*}\label{eqLie2 1}
d_\g([x,a])=[x,d_\g (a)], \quad [d_\g(a),b]=[a,d_\g (b)],
\end{equation*}
\begin{equation*}\label{eqLie2 2}
[x,[y,z]]+[y,[z,x]]+[z,[x,y]]=0,
\quad
[x,[y,a]]+[y,[a,x]]+[a,[x,y]]=0,
\end{equation*}
for all $x,y,z\in \g_{0},\  a,b\in \g_{1}$. We also refer to a strict Lie 2-algebra as a quadruple $(\g_{0},\g_{1},d_\g,[\cdot,\cdot])$.
\end{defi}
\begin{defi}{\rm(\cite{Baez})}
Let $\g=(\g_{0},\g_{1},d_{\g},[\cdot,\cdot]_{\g})$ and $\h= (\h_{0},\h_{1},d_{\h},[\cdot,\cdot]_{\h})$ be two strict Lie 2-algebras. A {\bf homomorphism} $\phi:\g \lon \h$ is a triple $(\phi_{0},\phi_{1},\phi_{2})$ of linear maps where $\phi_{0}:\g_{0}\lon \h_{0},\  \phi_{1}:\g_{1}\lon \h_{1}$ and $\phi_{2}:\wedge^2\g_{0}\lon \h_{1}$, such that 
\begin{eqnarray}
\label{eqLie2hom1}\phi_{0}\circ d_{\g}&=&d_{\h}\circ \phi_{1},\\
\label{eqLie2hom2}\phi_{0}([x,y]_{\g})-[\phi_{0}(x),\phi_{0}(y)]_{\h}&=&d_{\h}(\phi_{2}(x,y)),\\
\label{eqLie2hom3}\phi_{1}([x,a]_{\g})-[\phi_{0}(x),\phi_{1}(a)]_{\h}&=&\phi_{2}(x,d_{\g}(a)),\\
\label{eqLie2hom4}
[\phi_{0}(x),\phi_{2}(y,z)]_{\h}+[\phi_{0}(y),\phi_{2}(z,x)]_{\h}&+&[\phi_{0}(z),\phi_{2}(x,y)]_{\h}\\
\nonumber=\phi_{2}([x,y]_{\g},z)&+&\phi_{2}([y,z]_{\g},x)+\phi_{2}([z,x]_{\g},y),
\end{eqnarray}
for all $x,y,z\in \g,\  a\in \h$.
\end{defi}

\begin{ex}\label{impex1}
Given a Lie algebra $(\g,[\cdot,\cdot]_{\g})$, we obtain the following strict Lie 2-algebras.

{\bf (1)} Let $(\g_0, \g_1, d_\g, [\cdot, \cdot])=(\g,0,0,[\cdot,\cdot]_{\g})$. Then $(\g,0,0,[\cdot,\cdot]_{\g})$ is a strict Lie 2-algebra. 

{\bf (2)} Let $\g_{0}=\Der(\g),\  \g_{1}=\g$ and 
\begin{align*}
    d_\g=\ad&:\g\lon \Der (\g),& d_\g(x)&=\ad(x),&\\
    [\cdot,\cdot]&:\Der(\g)\wedge \Der(\g) \lon \Der (\g),&[D,D']&=D\circ D'-D'\circ D,&\\
    [\cdot,\cdot]&:\Der(\g)\wedge \g \lon \g,&[D,x]&=D(x),&
\end{align*}
for all $D,D'\in \Der(\g),  x\in \g$. Then $(\Der(\g),\g,\ad,[\cdot,\cdot])$ is a strict Lie 2-algebra. 
\end{ex}

\begin{defi}{\rm(\cite{LLS})}
Let $\g=(\g_{0},\g_{1},d_\g,[\cdot,\cdot])$ be a strict Lie 2-algebra. A {\bf strict derivation} on $\g$ is a pair of linear maps $\huaD=(\huaD_{0},\huaD_{1})$, where $\huaD_{0}:\g_{0}\lon \g_{0},\  \huaD_{1}:\g_{1}\lon \g_{1}$, such that 
\begin{eqnarray*}\label{eq2der}
\huaD_{0}\circ d_\g&=&d_\g \circ \huaD_{1},\\
\huaD_{0}([x,y])&=&[\huaD_{0}(x),y]+[x,\huaD_{0}(y)],\\
\huaD_{1}([x,a])&=&[\huaD_{0}(x),a]+[x,\huaD_{1}(a)],
\end{eqnarray*}
for all $x,y\in \g_{0},\  a\in \g_{1}$.
We denote a strict Lie 2-algebra with a strict derivation by $(\g,\huaD)$, called a {\bf \LtoD} pair. 
\end{defi}

\begin{defi}{\rm (\cite{LW})}
A {\bf homomorphism} $(\phi,\theta_{\phi}):(\g,\huaD)\lon (\h,\huaK)$ between two \LtoD pairs is a Lie 2-algebra homomorphism $\phi=(\phi_{0},\phi_{1},\phi_{2}):\g\lon\h$ with a linear map $\theta_{\phi}:\g_{0}\lon\h_{1}$ such that 
\begin{eqnarray}
\label{eqLie2Derhom1}
\phi_{0}\circ \huaD_{0}-\huaK_{0}\circ \phi_{0}&=&d_{\h}\circ \theta_{\phi},\\
\label{eqLie2Derhom2}
\phi_{1}\circ \huaD_{1}-\huaK_{1}\circ \phi_{1}&=&\theta_{\phi}\circ d_{\g},\\
\label{eqLie2Derhom3}
\huaK_{1}(\phi_{2}(x,y))-\phi_{2}(\huaD_{0}x,y)-\phi_{2}(x,\huaD_{0}y)
&=&[\theta_{\phi}(x),\phi_{0}(y)]_{\h}+[\phi_{0}(x),\theta_{\phi}(y)]_{\h}-\theta_{\phi}([x,y]_{\g}),
\end{eqnarray}
for all $x,y\in \g$.
\end{defi}

\begin{defi}{\rm (\cite{LW})}
Let $(\phi,\theta_{\phi}),(\psi,\theta_{\psi}):(\g,\huaD)\lon(\h,\huaK)$ be two homomorphisms of \LtoD pairs. A {\bf 2-homomorphism} $\vartheta:(\phi,\theta_{\phi})\Longrightarrow (\psi,\theta_{\psi})$ is a linear map $\vartheta:\g_0\lon\h_1$ satisfying
\begin{eqnarray}\label{2homo1}
\psi_{0}-\phi_{0}&=&d_{\h}\circ \vartheta,\quad \psi_{1}-\phi_{1}=\vartheta\circ d_{\g},\\
\label{2homo2}
\psi_{2}(x,y)-\phi_{2}(x,y)&=&\vartheta([x,y]_{\g})-[\phi_{0}(x),\vartheta(y)]_{\h}-[\vartheta(x),\psi_{0}(y)]_{\h},\\
\label{eqLie2Der2hom}
\vartheta \circ \huaD_{0}-\huaK_{1}\circ \vartheta&=&\theta_{\psi}-\theta_{\phi},
\end{eqnarray}
for all $x,y\in \g_{0}$. We call two homomorphisms are {\bf homotopic} if there exists a 2-homomorphism between them.
\end{defi}

Clearly, the homotopy relation between homomorphisms of \LtoD pairs is an equivalence relation and is compatible with composition of homomorphisms. This allows us to define the homotopy category of the category of \LtoD pairs. See \cite{May} for more details of the homotopy category of a category.

\begin{defi}
Denote by $\mathtt{LD}$ the category of \LtoD pairs, and denote by $\mathtt{hLD}$ the corresponding homotopy category.
\end{defi}

\begin{ex}
Given a \LD pair $(\g,D)$, by direct calculation, $(D, 0)$ is a strict derivation on $(\g,0,0,[\cdot,\cdot]_{\g})$, we denote this \LtoD pair by $(\g,\huaD)$.
\end{ex}
Let $(\h, K)$ be a \LD pair. By Example \ref{impex1}, $(\h_0, \h_1, d_\h, [\cdot, \cdot])=(\Der(\h),\h,\ad,[\cdot,\cdot])$ is a strict Lie 2-algebra. Define $\huaK_{\Der}=(\huaK_{\Der, 0}, \huaK_{\Der, 1}): (\Der(\h),\h,\ad,[\cdot,\cdot])\lon (\Der(\h),\h,\ad,[\cdot,\cdot])$ by $\huaK_{\Der, 0}=\ad K,~~\huaK_{\Der, 1}=K$, then we have 
\begin{equation*}
\huaK_{\Der,0}\circ \ad(x)=\ad K(\ad (x))=[K,\ad(x)]=\ad(K(x))=\ad \circ \huaK_{\Der,1}(x),
\end{equation*}
\begin{eqnarray*}
\huaK_{\Der,0}([K',K''])&=&[K,[K',K'']]=[[K,K'],K'']+[K,[K',K'']]\\
&=&[\huaK_{\Der,0}(K'),K'']+[K',\huaK_{\Der,0}(K'')],
\end{eqnarray*}
and
\begin{eqnarray*}
\huaK_{\Der,1}([K',x])&=&K(K'(x))
=[K,K'](x)+K'(K(x))=[[K,K'],x]+[K',K(x)]\\
&=&[\huaK_{\Der,0}(K'),x]+[K',\huaK_{\Der,1}(x)]
\end{eqnarray*}
for all $K',K''\in \Der(\h)$ and $x\in \h$.
Thus, we have the following result.
\begin{pro}
With the above notations, $\huaK_{\Der}=(\ad K,K)$ is a strict derivation on the \LtoD pair
$(\Der(\h),\h,\ad,[\cdot,\cdot])$. Denote this \LtoD pair by $(\h_{\Der},\huaK_{\Der})$.
\end{pro}

By Theorem \ref{corres}, non-abelian extensions of $(\g, D)$ by $(\h, K)$ are classified by the second non-abelian cohomology $H_{nab}^2
(\g,D; \h,K)$. At the end of this section, we will establish the relationship between non-abelian extensions of \LD pairs and hHom sets between strict Lie 2-algebras with strict derivations via $H^2_{nab}(\g,D; \h,K)$.
\begin{pro}\label{2coas2homo}
Let $(\g,D)$ and $(\h,K)$ be two \LD pairs. Then there exists a one-to-one correspondence between the set of non-abelian 2-cocycles $Z^2_{nab}(\g,D;\h,K)$ and $\rm{Hom}_{\mathtt{LD}}((\g,\huaD),(\h_{\Der},\huaK_{\Der}))$.  
\end{pro}
\begin{proof}
Suppose $(\varrho,\omega,\chi)\in Z^2_{nab}(\g,D;\h,K)$, define $(\phi_{0},\phi_{1},\phi_{2},\theta_{\phi})=(\varrho,0,-\omega,-\chi)$. Then   
\begin{equation*}
\phi_{0}\circ d_{\g}=\phi_{0} \circ 0=0\circ \ad =d_{\h}\circ \phi_{1}.
\end{equation*}
Hence, (\ref{eqLie2hom1}) is satisfied. Moreover, by (\ref{nc2}), it follows that
\begin{eqnarray*}
\phi_{0}([x,y]_{\g})-[\phi_{0}(x),\phi_{0}(y)]_{\h_{\Der}}&=&\varrho([x,y]_{\g})-[\varrho(x),\varrho(y)]_{\Der(\h)}=\varrho([x,y]_{\g})-\varrho(x)\varrho(y)+\varrho(y)\varrho(x)\\
&=&-\ad \omega(x,y)=d_{\h}\phi_{2}(x,y),
\end{eqnarray*}
for all $x,y\in \g$. Thus (\ref{eqLie2hom2}) is satisfied as well. By the fact that $d_{\g}$ and $\phi_{0}$ are zero maps, equations (\ref{eqLie2hom3}) and (\ref{eqLie2hom4}) are satisfied automatically. Therefore, the triple $(\phi_{0},\phi_{1},\phi_{2})$ is a Lie 2-algebra homomorphism between $\g=(\g,0,0,[\cdot,\cdot]_{\g})$ and $(\Der(\h),\h,\ad,[\cdot,\cdot])$. 

By equation (\ref{nc3}), we have 
\begin{eqnarray*}
\phi_{0}\circ \huaD_{0}(x)-\huaK_{\Der,0}\circ \phi_{0}(x)
=\varrho(D(x))-[K,\varrho (x)]%=\varrho(D(x))-K\circ \varrho(x)+\varrho(x)\circ K\\
=-\ad \chi(x)=d_{\h}\circ \theta_{\h}(x),
\end{eqnarray*}
for all $x\in \g$. Thus equation (\ref{eqLie2Derhom1}) is satisfied. Moreover, since $d_{\g}$ and $\phi_{0}$ are zeros, the condition (\ref{eqLie2Derhom2}) holds automatically. By \eqref{nc4}, we further obtain 
\begin{eqnarray*}
&&\huaK_{\Der, 1}(\phi_{2}(x,y))-\phi_{2}(\huaD_{0}(x),y)-\phi_{2}(x,\huaD_{0}(y))-[\theta_{\phi}(x),\phi_{0}(y)]_{\h}-[\phi_{0}(x),\theta_{\phi}(y)]_{\h}+\theta_{\phi}([x,y]_{\g})\\
%&=&K(-\omega(x,y))-(-\omega(D(x),y))-(-\omega(x,D(y)))-[-\chi(x),\varrho(y)]-[\varrho(x),-\chi(y)]+(-\chi([x,y]_{\g}))\\
&=&-\big( K\omega(x,y)-\omega(D(x),y)-\omega(x,D(y))-[\chi (x),\varrho (y)]-[\varrho(x),\chi (y)]+\chi([x,y]_{\g})\big)\\
%&=&-\big( K\omega(x,y)-\omega(D(x),y)-\omega(x,D(y))+\varrho(y)\chi(x)-\varrho(x)\chi(y)+\chi([x,y]_{\g})\big)\\
&=&0,
\end{eqnarray*}
for all $x,y\in \g$, which means that \eqref{eqLie2Derhom3} is satisfied. Therefore, $(\phi_{0},\phi_{1},\phi_{2},\theta_{\phi})=(\varrho,0,-\omega,-\chi):(\g,\huaD)\lon(\h_{\Der},\huaK_{\Der})$ is a homomorphism of \LtoD pairs.

Conversely, given a homomorphism $(\phi_{0},\phi_{1},\phi_{2},\theta_{\phi}):(\g,\huaD)\lon(\h_{\Der},\huaK_{\Der})$, define $(\varrho,\omega,\chi)=(\phi_{0},-\phi_{2},-\theta_{\phi})$. Similarly as above, by direct calculation, one can check that $(\varrho,\omega,\chi)=(\phi_{0},-\phi_{2},-\theta_{\phi})$ is a non-abelian 2-cocycle of $(\g,D)$ with values in $(\h,K)$.

The two processes mentioned above are inverse to each other, in other words, we obtain a one-to-one correspondence between $Z^2_{nab}(\g, D; \h, K)$ and $\Hom_{\mathtt{LD}}((\g,\huaD),(\h_{\Der},\huaK_{\Der}))$.
\end{proof}

\begin{thm}\label{0000}
Let $(\g,D)$ and $(\h,K)$ be two \LD pairs. Then there is a one-to-one correspondence between $H^2_{nab}(\g,D; \h, K)$ and $\rm{hHom}_{\mathtt{LD}}((\g,\huaD),(\h_{\Der},\huaK_{\Der}))$. Moreover,
\begin{equation*}
{\rm{\bf Ext_{nab}}}(\g, D; \h, K)\cong \rm{hHom}_{\mathtt{LD}}((\g,\huaD),(\h_{\Der},\huaK_{\Der})).
\end{equation*}
\end{thm}
\begin{proof}
Suppose that $(\varrho,\omega,\chi)$ and $(\varrho',\omega',\chi')$ are two equivalent non-abelian 2-cocycles, and that the equivalence is given by $\tau:\g\lon\h$. By Proposition \ref{2coas2homo}, the maps $(\phi_{0},\phi_{1},\phi_{2},\theta_{\phi})=(\varrho,0,-\omega,-\chi)$ and 
$(\psi_{0},\psi_{1},\psi_{2},\theta_{\psi})=(\varrho',0,-\omega',-\chi')$ are homomorphisms of Lie2Dear pairs. Define $\vartheta=\tau:\g\lon\h$, by (\ref{eq11}), we have
\begin{equation*}
\psi_{0}(x)-\phi_{0}(x)=\varrho (x)-\varrho'(x)=\ad \circ \tau(x)=d_{\h}\circ \vartheta(x),
\end{equation*}
for all $x\in \g$, thus the first equation of (\ref{2homo1}) holds. The second of (\ref{2homo1}) holds automatically as $d_{\g}, \psi_{1}$ and $\phi_{1}$ are zero maps. By equation (\ref{eq22}), we obtain
\begin{eqnarray*}
\psi_{2}(x,y)-\phi_{2}(x,y)&=&(-\omega)(x,y)-(-\omega')(x,y)\\
&=&\varrho'(x)\tau(y)-\varrho'(y)\tau(x)+[\tau(x), \tau(y)]_\h-\tau([x, y]_\g)\\
&=&\vartheta([x,y]_{\g})-[\phi_{0}(x),\vartheta (y)]-[\vartheta (x),\psi_{0}(y)],
\end{eqnarray*}
for all $x,y\in \g$, which means that equation (\ref{2homo2}) holds. Therefore, $(\phi_{0},\phi_{1},\phi_{2})$ and $(\psi_{0},\psi_{1},\psi_{2})$ are homotopic as homomorphisms of Lie 2-algebras. By (\ref{eq33}), we obtain
\begin{eqnarray*}
\vartheta \circ \huaD_{0}-\huaK_{1}\circ \vartheta=\tau\circ D-K\circ \tau=(-\chi')-(-\chi)=\theta_{\psi}-\theta_{\phi}.
\end{eqnarray*}
Hence $\vartheta=\tau$ is a 2-homomorphism from $(\phi_{0},\phi_{1},\phi_{2},\theta_{\phi})$ to 
$(\psi_{0},\psi_{1},\psi_{2},\theta_{\psi})$. In other words, $(\phi_{0},\phi_{1},\phi_{2},\theta_{\phi})$ and 
$(\psi_{0},\psi_{1},\psi_{2},\theta_{\psi})$ are homotopic. The other direction follows by a similar argument. Thus, $H^2_{nab}(\g,\h)\cong\rm{hHom}_{\mathtt{LD}}((\g,\huaD),(\h_{\Der},\huaK_{\Der}))$. By Theorem \ref{corres},
${\rm{\bf Ext_{nab}}}(\g, D; \h, K)\cong \rm{hHom}_{\mathtt{LD}}((\g,\huaD),(\h_{\Der},\huaK_{\Der}))$.
\end{proof}
According to Proposition \ref{2coas2homo} and Theorem \ref{0000}, we have the following result.
\begin{thm}
    Let $(\g,D)$ and $(\h,K)$ be two \LD pairs. Denote by $\mathtt{LieDer}$ the category of \LD pairs. Then the functor
    \begin{equation*}
       \mathtt{LieDer}\lon\mathtt{Set},\quad  (\g,D)\mapsto {\rm{\bf Ext_{nab}}}(\g,D;\h,K),
    \end{equation*}
    is representable in the homotopy category $\mathtt{hLD}$ of \LtoD pairs, which is represented by the \LtoD pair $(\h_{\Der},\huaK_{\Der})$. Here, $\mathtt{LieDer}$ is viewed as subcategory of $\mathtt{LD}$.
\end{thm}

\section{Non-abelian extensions of LieDer pairs and $(\g, D)$-kernels}
In this section, we introduce the notion of the $(\g, D)$-kernel, which will be used to construct the cohomology groups of \LD pairs and to study non-abelian extensions of \LD pairs. We recall the representations and cohomology theory of \LD pairs firstly.
\begin{defi}{\rm(\cite{TFS})}
Let $(\g, D)$ be a {\rm LieDer} pair. A {\bf representation} of $(\g, D)$ on a vector space $V$ is a pair $(\rho, T)$, where $\rho:\g\lon\gl(V)$ is a representation of the Lie algebra $\g$ on the vector space $V$ and $T\in\gl(V)$ such that the following equation holds:
\begin{equation}\label{ldrep}
    T(\rho(x)u)=\rho(D(x))u+\rho(x)T(u), \quad \forall x\in\g, u\in V.
\end{equation}
\end{defi}
Let $(\g, D)$ be a {\rm LieDer} pair and $(\rho, T)$ be a representation of $(\g, D)$ on a vector space $V$. Denote by $\dM_\rho^{CE}:\Hom(\wedge^{n}\g, V)\lon\Hom(\wedge^{n+1}\g, V)$ the Chevalley-Eilenberg coboundary operator of the Lie algebra $\g$ with coefficients in the representation $(V, \rho)$. Define the $1$-cochains $C^{1}_{\mathrm{LieDer}}(\g, V)$ to be $\Hom(\g, V)$. For $n\geq 2$, define the space of $n$-cochains $C^{n}_{\mathrm{LieDer}}(\g, V)$ by
$$
C^{n}_{\mathrm{LieDer}}(\g, V)=\Hom(\wedge^{n}\g, V)\oplus\Hom(\wedge^{n-1}\g, V).
$$
For $n\geq 1$, define an operator $\delta:C^{n}_{\mathrm{LieDer}}(\g, V)\lon C^{n}_{\mathrm{LieDer}}(\g, V)$ by
\begin{equation}\label{deeqsigma}
\delta(\omega_n)=\sum_{i=1}^{n}\omega_n\circ(\Id\otimes\cdots\otimes D\otimes\cdots\otimes\Id)-T\circ\omega_n.
\end{equation}
For $n=1$, define the coboundary operator $\partial_\rho:C^{1}_{\mathrm{LieDer}}(\g, V)\lon C^{2}_{\mathrm{LieDer}}(\g, V)$ by
$$
\partial_\rho(\omega_1)=(\dM_\rho^{CE}\omega_1, -\delta(\omega_1)). 
$$
For $n\geq 2$, define the coboundary operator $\partial_\rho:C^{n}_{\mathrm{LieDer}}(\g, V)\lon C^{n+1}_{\mathrm{LieDer}}(\g, V)$ by
\begin{equation}\label{deficeo}
\partial_\rho(\omega_n, \chi_{n-1})=\Big(\dM_\rho^{CE}\omega_n, \dM_\rho^{CE}\chi_{n-1}+(-1)^n\delta(\omega_n)\Big)
\end{equation}
where $\omega_n\in\Hom(\wedge^{n}\g, V)$ and $\chi_{n-1}\in\Hom(\wedge^{n-1}\g, V)$. Then $(C^{*}_{\mathrm{LieDer}}(\g, V), \partial_\rho)$ is a complex\cite{TFS}.
\begin{defi}{\rm(\cite{TFS})}
The cohomology of the complex $(C^{*}_{\mathrm{LieDer}}(\g, V), \partial_\rho)$ is called the {\bf cohomology} of the \LD pair $(\g, D)$ with values in $(V, T)$. Denote the set of $n$-cocycles by $Z^{n}_{\LD}(\g,D; \h,K)$, the $n$-th cohomology group by $H^{n}_{\LD}(\g, D;\h,K)$.
\end{defi}
\subsection{The relation between non-abelian extensions of LieDer pairs and $(\g, D)$-kernels}
Let $\g$ be a Lie algebra. Denote by $\Out(\g)=\Der(\g)/\ad(\g)$ the set of outer derivations of $\g$, then $(\Out(\g), [\cdot, \cdot]_{\Out(\g)})$ is a Lie algebra as $\ad(\g)$ is the ideal of $(\Der(\g), [\cdot, \cdot]_{\Der(\g)})$. Assume $D\in\Der(\g)$, denote by $\bar{D}$ the equivalence class of $D$ in $\Out(\g)$.

\begin{defi}
    Let $(\g,D)$ and $(\h,K)$ be two \LD pairs. A {\bf $(\g,D)$-kernel} for $(\h, K)$ is a Lie algebra homomorphism $\frkk: \g\lon \Out(\h)$  such that 
    \begin{equation}\label{gdk0}
        [\overline{K},  \frkk(x)]_{\Out(\h)}=\frkk(D(x)),\quad \forall x\in\g,
    \end{equation}
    where $\overline{K}$ is the outer derivation induced by $K$.
\end{defi}
Denote by $\frkz(\h)$ the center of a Lie algebra $\h$. It is well known that every derivation $K:\h\lon\h$ preserves $\frkz(\h)$. Let $\nu\in\Out(\h)$ and choose $T\in\Der(\h)$ such that $\overline{T}=\nu$. Then the linear map $$\widetilde{\nu}:\frkz(\h)\lon\frkz(\h), \quad
\widetilde{\nu}(u)=T(u), \quad ~~\forall~~ u\in\frkz(\h),
$$
is independent of the choice of $T$. Hence, without ambiguity, $\nu\in\Out(\h)$ can be regarded as a linear map on $\frkz(\h)$.
Given a $(\g,D)$-kernel $\frkk$ for $(\h, K)$, then $\frkk(x)\in\Out(\h)$ for all $x\in\g$, which means that we can obtain a linear map $\rho_{\frkk}:\g\lon\gl(\frkz(\h))$ given by $$\rho_{\frkk}(x)(u)=\frkk(x)(u),\quad \forall x\in\g,\ u\in\frkz(\h).$$
\begin{pro}\label{repkk}
With the above notations,  $\rho_{\frkk}:\g\lon\gl(\frkz(\h))$ is a representation of $(\g,D)$ on $(\frkz(\h),K)$. The representation $\rho_{\frkk}$ is called the induced representation  of the $(\g,D)$-kernel $\frkk$.    
\end{pro}
\begin{proof}
Since
\begin{eqnarray*}
    \rho_{\frkk}([x.y]_{\g})u&=&\frkk([x,y]_{\g})(u)=[\frkk(x),\frkk(y)]_{\Out(\h)}(u)%=\frkk(x)(\frkk(y)(u))-\frkk(y)(\frkk(x)(u))\\
    =\rho_{\frkk}(x)(\rho_{\frkk}(y)u)-\rho_{\frkk}(y)(\rho_{\frkk}(x)u),
\end{eqnarray*}
for all $x\in\g,\ u\in \frkz(\h)$, then $\rho_{\frkk}:\g\lon\gl(\frkz(\h))$ is a Lie algebra representation. By equation (\ref{gdk0}), we have
\begin{eqnarray*}
   [\overline{K}, \frkk(x)]_{\Out(\g)}(u)&=&K(\rho_{\frkk}(x)u)-\rho_{\frkk}(x)K(u)=\frkk(D(x))(u)=\rho_{\frkk}(D(x))u,
\end{eqnarray*}
for all $x\in \g,\ u\in \frkz(\h)$. Thus $\rho_{\frkk}:\g\lon\gl(\frkz(\h))$ is a representation of $(\g,D)$ on $(\frkz(\h),K)$.
\end{proof}
Let $(\hat{\g}, \hat{D})$ be a non-abelian extension of $(\g, D)$ by $(\h, K)$. By Theorem \ref{corres}, there exists a $[(\varrho, \omega, \chi)]\in H^{2}_{nab}(\g, D;\h, K)$ corresponding to $(\hat{\g}, \hat{D})$. For the linear map $\varrho:\g\lon\Der(\h)$, we define a linear map $\upsilon:\g\lon\Out(\h)$ by
$$
\upsilon(x)=\overline{\varrho(x)}, \quad~~\forall x\in\g,
$$
where $\overline{\varrho(x)}$ denotes the equivalence class of $\varrho(x)$ in $\Out(\g)$. Moreover, if $[(\varrho', \omega', \chi')]=[(\varrho, \omega, \chi)]$, then there exists a linear map $\tau:\g\lon\h$ such that $\varrho-\varrho'=\ad\tau$, which implies that $\overline{\varrho(x)}=\overline{\varrho'(x)}$ for all $x\in\g$. Thus, the linear map $\upsilon:\g\lon\Out(\h)$ is well defined.
 
\begin{pro}\label{barrho}
   With the above notations, the map $\upsilon:\g\lon\Out(\h)$ is a $(\g,D)$-kernel for $(\h,K)$. Moreover, if $(\hat{\g}, \hat{D})$ and $(\tilde\g, \tilde{D})$ are two isomorphic non-abelian extensions of $(\g,D)$ by $(\h,K)$, then the corresponding $(\g,D)$-kernels for $(\h, K)$ are the same.
   \end{pro}
    \begin{proof}
       By \eqref{nc1} and \eqref{nc3}, we obtain the following equations
    \begin{eqnarray*}\label{gdk1}\label{gdk2}
            \varrho([x,y]_{\g})=[\varrho(x),\varrho (y)]_{\Der(\g)}-\ad \omega(x,y),\quad
        \text{and}\quad
            K\circ \varrho(x)=\varrho(Dx)+\varrho(x)\circ K+\ad\chi(x),
        \end{eqnarray*}
        for all $x,y\in \g$. Thus, for all $x, y\in\g$, we have
        \begin{eqnarray*}
\upsilon([x,y]_{\g})&=&\overline{\varrho([x,y]_{\g})}=\overline{[\varrho(x),\varrho(y)]_{\Der(\h)}-\ad \omega(x,y)}=\overline{[\varrho (x),\varrho(y)]_{\Der(\h)}}=[\overline{\varrho (x)},\overline{\varrho (y)}]_{\Out(\h)}\\
&=&[\upsilon(x), \upsilon(y)]_{\Out(\h)},
        \end{eqnarray*}
        and
        \begin{eqnarray*}
    \upsilon(D(x))&=&\overline{\varrho(D(x))}=\overline{K\circ \varrho(x)-\varrho(x)\circ K+\ad \chi(x)}=\overline{K\circ\varrho(x)-\varrho(x)\circ K}\\
            &=&\overline{K}\circ \upsilon(x)-\upsilon(x)\circ\overline{K}=[\overline{K}, \upsilon(x)]_{\Out(\h)},
        \end{eqnarray*}
        which means that $\upsilon$ is a $(\g,D)$-kernel for $(\h, K)$.

Denote by $(\varrho,\omega,\chi)$ and $(\varrho',\omega',\chi')$ the corresponding non-abelian 2-cocycles of $(\hat\g, \hat{D})$ and $(\tilde\g, \tilde{D})$ respectively. Then there exists a linear map $\tau:\g\lon\h$ such that 
        \begin{equation*}
            \varrho'(x)-\varrho(x)=\ad \tau (x),
        \end{equation*}
       which means
        $
\overline{\varrho'(x)}=\overline{\varrho(x)},
        $
        for all $x\in\g$. Thus, the corresponding $(\g,D)$-kernels coincide.
    \end{proof}

For a given $(\g,D)$-kernel $\frkk$ for $(\h, K)$, denote by $\mathrm{Ext}(\g,D; \h, K)_{\frkk}$ the set of isomorphism classes of non-abelian extensions of $(\g,D)$ by $(\h,K)$ whose associated $(\g,D)$-kernel is $\frkk$. The set ${\rm{\bf Ext_{nab}}}(\g, D;\h, K)$ of all isomorphism classes of non-abelian extensions of $(\g,D)$ by $(\h,K)$ is then the disjoint union of the sets $\mathrm{Ext}(\g,D;\h,K)_{\frkk}$, parametrized by $\frkk$. Hence, the classification of all non-abelian extensions of $(\g, D)$ by $(\h, K)$ reduces to the following two questions.
\begin{que}\label{que12}
Let $(\g, D)$ and $(\h, K)$ be \LD pairs.
\begin{itemize}
\item {\rm(i)}:
For a fixed $(\g,D)$-kernel $\frkk$ for $(\h, K)$, give a parameterization of $\mathrm{Ext}(\g,D;\h,K)_{\frkk}$. 
\item {\rm(ii)}:
Find a method to decide for which $(\g,D)$-kernels the set $\mathrm{Ext}(\g,D;\h,K)_{\frkk}$ is nonempty.
\end{itemize}
\end{que}

\begin{defi}
    Let $\frkk:\g\lon\Out(\h)$ be a $(\g,D)$-kernel for $(\h, K)$. We say that $\frkk$ is {\bf integrable} if $\mathrm{Ext}(\g,D;\h,K)_{\frkk}\neq \emptyset$.
\end{defi}
Let $\frkk:\g\lon\Out(\h)$ be an integrable $(\g,D)$-kernel for $(\h, K)$. By Proposition \ref{repkk}, we have that $\rho_{\frkk}:\g\lon\gl(\frkz(\h))$ is a representation of $(\g,D)$ on $(\frkz(\h),K)$, where
$$
\rho_{\frkk}(x)u=\frkk(x)u, \quad \forall u\in\frkz(\h),  ~~x\in\g.
$$
Denote by $H_{\LD}^{n}(\g, D; \frkz(\h), K)_\frkk$ the $n$-th cohomology group of $(\g, D)$ with coefficients in the representation $(\rho_\frkk, K)$. Therefore, for the first problem, we have the following answer.
\begin{thm}\label{gdk}
With the above notations, the set $\mathrm{Ext}(\g,D;\h,K)_{\frkk}$ is isomorphic to $H^2_{\LD}(\g,D;\frkz(\h),K)_{\frkk}$. Moreover, we have 
    \begin{equation*}
      {\rm{\bf Ext_{nab}}}(\g,D;\h,K)=\bigsqcup_{\text{Integrable}~~(\g,D)\text{-kernels}}H^2_{\LD}(\g,D; \frkz(\h), K)_{\frkk}.
    \end{equation*}
    \end{thm}
    \begin{proof}
Fix an element $[(\hat{\g}, \hat{D})]$ in $\rm{Ext}(\g, D;\h,K)_\frkk$, it means that $[(\hat{\g}, \hat{D})]$ is the isomorphism class of $(\hat{\g}, \hat{D})$, where $(\hat{\g}, \hat{D})$ is a non-abelian extension of $(\g,D)$ by $(\h,K)$
whose associated $(\g, D)$-kernel is $\frkk$. 
By Theorem \ref{corres}, the non-abelian extensions of $(\g, D)$ by $(\h, K)$ are classified by the second non-abelian cohomology $H^2_{nab}(\g,D; \h,K)$. Let $[(\varrho, \omega, \chi)]$ be the element in $H^{2}_{nab}(\g, D;\h,K)$ corresponding to $[(\hat{\g}, \hat{D})]$. Now, let $[(\tilde{\g}, \tilde{D})]\in{\rm Ext}(\g, D;\h,K)_\frkk$ and $[(\tilde{\g}, \tilde{D})]\neq [(\hat{\g}, \hat{D})]$. Denote by $[(\tilde{\varrho}, \tilde{\omega}, \tilde{\chi})]$ the element in $H^{2}_{nab}(\g,D; \h,K)$ corresponding to $[(\tilde{\g}, \tilde{D})]$. Then $\overline{\varrho}=\overline{\tilde{\varrho}}=\frkk$, which means that there exists a linear map $\tau:\g\lon\h$ such that
$$\varrho(x)-\tilde{\varrho}(x)=\ad\tau(x), \quad \forall x\in\g.$$
Define linear maps $\omega^*:\wedge^2\g\lon\h$ and $\chi^*:\g\lon\h$ by
\begin{eqnarray}
\label{H1}\omega^*(x, y)&=&\tilde{\omega}(x, y)+\tilde{\varrho}(x)\tau(y)-\tilde{\varrho}(y)\tau(x)+[\tau(x), \tau(y)]_\h-\tau([x, y]_\g),\\
\label{H2}\chi^*(x)&=&\tilde{\chi}(x)+K(\tau(x))-\tau(D(x)),
\end{eqnarray}
for all $x, y\in\g$. Then $[(\tilde{\varrho}, \tilde{\omega}, \tilde{\chi})]=[(\varrho, \omega^*, \chi^*)]$. By \eqref{nc1} and \eqref{nc3}, it implies that
$$
[\omega(x, y)-\omega^*(x, y), u]_\h=0,\quad [\chi(x)-\chi^*(x), u]_\h=0,
$$
for all $x, y\in\g, u\in\h$,
which means that the map $\eta=\omega-\omega^*\in\Hom(\wedge^2\g, \frkz(\h))$ and $\theta=\chi-\chi^*\in\Hom(\g, \frkz(\h))$. By \eqref{nc2} and \eqref{nc4}, we have
$$
\dM^{CE}_{\rho_\frkk}(\eta)=0, \quad \dM^{CE}_{\rho_\frkk}(\theta)+\delta(\eta)=0,
$$
thus, $\partial_{\rho_\frkk}(\eta, \theta)=0$, i.e.,$(\eta, \theta)$ is a $2$-cocycle of the complex $(C^{*}_{\mathrm{LieDer}}(\g, \frkz(\h)), \partial_{\rho_\frkk})$. 

Suppose that $[(\varrho', \omega', \chi')]=[(\varrho, \omega, \chi)]$ and $
[(\tilde{\varrho}', \tilde{\omega}', \tilde{\chi}')]=[(\tilde{\varrho}, \tilde{\omega}, \tilde{\chi})]$ such that $\overline{\varrho}=\overline{\varrho'}=\overline{\tilde{\varrho}}=\overline{\tilde{\varrho}'}=\frkk$. Then there exist linear maps $\tau', \tilde{\tau}':\g\lon\h$ such that $(\varrho', \omega', \chi')$ is equivalent to $(\varrho, \omega, \chi)$, and $(\tilde{\varrho}', \tilde{\omega}', \tilde{\chi}')$ is equivalent to $(\tilde{\varrho}, \tilde{\omega}, \tilde{\chi})$.
Moreover, there exists $\tau^{*}{'}:\g\lon\h$ such that
\emptycomment{
\begin{eqnarray*}
\left\{\begin{array}{rcl}
\varrho(x)u-\varrho'(x)u&=&[\tau'(x), u]_\h,\\
\omega(x, y)-\omega'(x, y)&=&\varrho'(x)\tau'(y)-\varrho'(y)\tau'(x)+[\tau'(x), \tau'(y)]_\h-\tau'([x, y]_\g),\\
\chi(x)-\chi'(x)&=&K(\tau'(x))-\tau'(D(x)),
\end{array}\right.
\end{eqnarray*}
\begin{eqnarray*}
\left\{\begin{array}{rcl}
\tilde{\varrho}(x)u-\tilde{\varrho}'(x)u&=&[\tilde{\tau}'(x), u]_\h,\\
\tilde{\omega}(x, y)-\tilde{\omega}'(x, y)&=&\tilde{\varrho}'(x)\tilde{\tau}'(y)-\tilde{\varrho}'(y)\tilde{\tau}'(x)+[\tilde{\tau}'(x), \tilde{\tau}'(y)]_\h-\tilde{\tau}'([x, y]_\g),\\
\tilde{\chi}(x)-\tilde{\chi}'(x)&=&K(\tilde{\tau}'(x))-\tilde{\tau}'(D(x)),
\end{array}\right.
\end{eqnarray*}
for all $x, y\in\g, u\in\h$, and}
$$\varrho'(x)-\tilde{\varrho}'(x)=\ad\tau^{*}{'}(x), \quad\forall x\in\g.$$ Then $[(\varrho', \omega^{*}{'}, \chi^{*}{'})]=[(\tilde{\varrho}', \tilde{\omega}', \tilde{\chi}')]$, where
\begin{eqnarray}
\label{H3}\omega^{*}{'}(x, y)&=&\tilde{\omega}'(x, y)+\tilde{\varrho}'(x)\tau^{*}{'}(y)-\tilde{\varrho}'(y)\tau^{*}{'}(x)+[\tau^{*}{'}(x), \tau^{*}{'}(y)]_\h-\tau^{*}{'}([x, y]_\g),\\
\label{H4}\chi^{*}{'}(x)&=&\tilde{\chi}'(x)+K(\tau^{*}{'}(x))-\tau^{*}{'}(D(x)).
\end{eqnarray}
Moreover, the maps $\eta'=\omega'-\omega^{*}{'}$ and $\theta'=\chi'-\chi^{*}{'}$ satisfy $\partial_{\rho_\frkk}(\eta', \theta')=0$. By \eqref{H1}, \eqref{H2}, \eqref{H3} and \eqref{H4}, it implies that
\begin{eqnarray*}
(\eta-\eta')(x, y)=\omega(x, y)-\omega'(x, y)-\omega^{*}(x, y)+\omega^{*}{'}(x, y) =\dM^{CE}_{\rho_\frkk}(\tau'+\tau^{*}{'}-\tilde{\tau}{'}-\tau)(x, y),
\end{eqnarray*}
and 
\begin{eqnarray*}
(\theta-\theta')(x)=\chi(x)-\chi'(x)-\chi^*(x)+\chi^{*}{'}(x)=-\delta(\tau'+\tau^{*}{'}-\tilde{\tau}'-\tau)(x),
\end{eqnarray*}
which means that $(\eta, \theta)-(\eta', \theta')=\partial_{\rho_{\frkk}}(\tau'+\tau^{*}{'}-\tilde{\tau}'-\tau)$, i.e.,$[(\eta, \theta)]=[(\eta', \theta')]\in H^{2}(\g, \frkz(\h))_{\frkk}$. Therefore the map  
$\Xi:{\rm Ext}(\g,D; \h, K)_\frkk\lon H_{\LD}^{2}(\g,D; \frkz(\h),K)_\frkk$ defined by
$$
\Xi([(\tilde{\g}, \tilde{D})])=[(\eta, \theta)], 
$$
is well defined, and $\Xi([(\hat{\g}, \hat{D})])=0$.

Assume that $\Xi([(\tilde{\g}, \tilde{D})])=0,$ there exists a linear map $l:\g\lon\frkz(\h)$ such that $\partial_{\rho_{\frkk}}(l)=(\eta, \theta)$. By direct calculation, 
\emptycomment{
	Then
\begin{eqnarray*}
\left\{\begin{array}{rcl}
\varrho(x)u-\tilde{\varrho}(x)u&=&[\tau(x)+l(x), u]_\h,\\
\omega(x, y)-\tilde{\omega}(x, y)&=&\tilde{\varrho}(x)(\tau+l)(y)-\tilde{\varrho}(y)(\tau+l)(x)+[(\tau+l)(x), (\tau+l)(y)]_\h-(l+\tau)([x, y]_\g),\\
\chi(x)-\tilde{\chi}(x)&=&K((\tau+l)(x))-(\tau+l)(D(x)),
\end{array}\right.
\end{eqnarray*}
which means that }$[(\varrho, \omega, \chi)]=[(\tilde{\varrho}, \tilde{\omega}, \tilde{\chi})]$, i.e., $[(\tilde{\g}, \tilde{D})]=[(\hat{\g}, \hat{D})]$. Thus $\Xi$ is injective. 

For any $[(\eta, \theta)]\in H^2_{\LD}(\g,D; \frkz(\h),K)_\frkk$, i.e., $\partial_{\rho_\frkk}(\eta, \theta)=0$. By direct calculation, $(\varrho, \omega-\eta, \chi-\theta)$ is a non-abelian $2$-cocycle. Denote by $(\bar{\g}, \bar{D})$ the non-abelian extension of $(\g, D)$ by $(\h, K)$ corresponding to $(\varrho, \omega-\eta, \chi-\theta)$. We have $\Xi([(\bar{\g}, \bar{D})])=[(\eta, \theta)]$, which means that $\Xi$ is surjective. Thus,
the set $\mathrm{Ext}(\g,D;\h,K)_{\frkk}$ is isomorphic to $H^2_{\LD}(\g,D;\frkz(\g),K)_{\frkk}$, which implies 
\begin{equation*}
      {\rm{\bf Ext_{nab}}}(\g,D;\h,K)=\bigsqcup_{\text{Integrable}(\g,D)\text{-kernels}}H^2_{\LD}(\g,D;\frkz(\h),K)_{\frkk}.
    \end{equation*}
We finish the proof.
    \end{proof}

\begin{rmk}
    By the above Theorem, we see that if $\mathrm{Ext}(\g,D;\h,K)_{\frkk}$ is nonempty, then $\mathrm{Ext}(\g,D;\h,K)_{\frkk}$ is an affine space with the associated vector space $H^2(\g,\frkz(\h))_{\frkk}$.
\end{rmk}

\subsection{Integrable $(\g, D)$-kernels and third cohomology groups} In this subsection, we investigate the second question in Question \ref{que12}:~~to find a criterion for determining for which $(\g,D)$-kernels the set $\mathrm{Ext}(\g,\h)_{\frkk}$ is nonempty. 

Let $(\g, D)$ and $(\h, K)$ be \LD pairs, and $\frkk:\g\lon\Out(\h)$ be a $(\g,D)$-kernel for $(\h,K)$. We can always choose a linear map $\varrho:\g\lon\Der(\h)$ such that $\overline{\varrho}=\frkk$, where $\overline{\varrho}$ denotes the image of $\varrho$ in $\Out(\h)$. Since $\frkk:\g\lon\Out(\h)$ is a homomorphism and
$$
[\bar{K}, \frkk(x)]_{\Out(\h)}=\frkk(D(x)),
$$
it follows that 
\begin{equation*}
    [\varrho(x),\varrho(y)]_{\Der(\h)}-\varrho[x,y]_{\g}\in \ad(\h),\quad\text{and}\quad
    K\circ\varrho(x)-\varrho(D(x))-\varrho(x)\circ K\in \ad(\h),
\end{equation*}
for all $x, y\in\g$. Therefore, there exist linear maps $\omega\in C^2(\g,\h)$ and $\chi\in C^1(\g,\h)$  such that
\begin{equation}\label{1defiomega}
    \ad \omega(x,y)=[\varrho(x),\varrho(y)]_{\Der(\h)}-\varrho([x,y]_{\g}),
\end{equation}
for all $x,y\in\g$, and 
\begin{equation}\label{1defichi}
    \ad\chi(x)=K\circ\varrho(x)-\varrho(Dx)-\varrho(x)\circ K,
\end{equation}
for all $x\in\g$. However, in general, equations (\ref{nc2}) and (\ref{nc4}) may not be satisfied, in other words, $(\varrho,\omega,\chi)$ may not be a non-abelian 2-cocycle.

Given a linear map $\varrho:\g\lon\Der(\h)$ such that $\overline{\varrho}=\frkk$, we can define a formal coboundary operator $\dM^F_{\varrho}:\oplus_{i=1}^{+\infty}\Hom(\wedge^{i}{\g},\h)\lon\oplus_{i=1}^{+\infty}\Hom(\wedge^{i}{\g},\h)$ by
\begin{eqnarray*}\label{formald}
&&\dM^F_{\varrho}(\alpha_l)(x_1, \cdots, x_{l+1})\\
\nonumber&=&\sum_{i=1}^{l+1}(-1)^{i+1}\varrho(x_i)\alpha_l(x_1, \cdots, \hat{x}_i, \cdots, x_{l+1})+\sum_{i<j}(-1)^{i+j}\alpha_l(x_1, \cdots,\hat{x}_i, \cdots, \hat{x}_j, \cdots, x_{l+1}).
\end{eqnarray*}
Note that, in general, $(\oplus_{i=1}^{+\infty}\Hom(\wedge^{i}{\g},\h),\dM^F_{\varrho})$ is not a cochain complex, but when $\varrho$ is a representation of the Lie algebra $\g$,  $\dM^F_\varrho$ is the Chevalley-Eilenberg coboundary operator.
In the following, we will give necessary and sufficient conditions for the triple $(\varrho,\omega,\chi)$ satisfying equations (\ref{nc2}) and (\ref{nc4}) in terms of $\dM^F_\varrho$.

Firstly, we recall the definition of the cup product on $\oplus_{n\geq 1}\Hom(\wedge^n\g,\h)$, where $\g$ and $\h$ are Lie algebras.
\begin{defi}{\rm (\cite{NR1})}
    Let $(\g, [\cdot, \cdot]_\g)$ and $(\h, [\cdot, \cdot]_\h)$ be Lie algebras. Denote by $C(\g,\h)=\oplus_{n\geq 1}\Hom(\wedge^n\g,\h)$. The {\bf cup product} $[\cdot,\cdot]_{\smile}:C(\g, \h)\times C(\g, \h)\lon C(\g, \h)$ is defined as follows
    \begin{eqnarray}
        [\alpha,\beta]_{\smile}(x_{1},\cdots,x_{p+q}):=\sum_{\sigma\in S(p,q)} (-1)^\sigma[\alpha(x_{\sigma(1)},\cdots,x_{\sigma(p)}),\beta(x_{\sigma(p+1)},\cdots,x_{\sigma(p+q)})]_{\h},
    \end{eqnarray}
    where $\alpha\in C^p(\g,\h),\ \beta\in C^q(\g,\h)$ and $S(p,q)$ the set of all $(p,q)$-shuffles.
\end{defi}
In particular, for $p=1$ and $q=2$ in the above definition, we have 
\begin{equation*}
    [\alpha,\alpha]_{\smile}(x,y)=[\alpha(x),\alpha(y)]_{\h}-[\alpha(y),\alpha(x)]_{\h}=2[\alpha(x),\alpha(y)]_{\h},
\end{equation*}
and
\begin{equation*}
    [\alpha,\beta]_{\smile}(x,y,z)=[\alpha(x),\beta(y,z)]_{\h}+[\alpha(y),\beta(z,x)]_{\h}+[\alpha(z),\beta(x,y)]_{\h},
\end{equation*}
as well as
\begin{equation*}
    [\beta,\alpha]_{\smile}(x,y,z)=[\beta(x,y),\alpha(z)]_{\h}+[\beta(y,z),\alpha(x)]_{\h}+[\beta(z,x),\alpha(y)]_{\h}=-[\alpha,\beta]_{\smile}.
\end{equation*}
By the {\rm Jacobi} identity, we obtain
\begin{equation}\label{rrr0}
[\alpha,[\alpha,\alpha]_{\smile}]_{\smile}=0.
\end{equation}
\begin{lem}
    For any $\varepsilon \in \Hom(\g,\h)$, we have 
    \begin{eqnarray}\label{dsmile}
        \dM^F_{\varrho}[\varepsilon,\varepsilon]_{\smile}&=&2[\dM^F_{\varrho}\varepsilon,\varepsilon]_{\smile}=-2[\varepsilon,\dM^F_{\varrho}\varepsilon]_{\smile},\\
    \label{deltasmile}
        \delta([\varepsilon,\varepsilon]_{\smile})&=&[\delta(\varepsilon),\varepsilon]_{\smile}+[\varepsilon,\delta(\varepsilon)]_{\smile}=2[\varepsilon,\delta(\varepsilon)]_{\smile},
    \end{eqnarray}
    where $\delta$ is given by \eqref{deeqsigma}.
    \end{lem}
 \begin{proof}
    By \cite[Lemma 7.6.15]{HN}, we have 
    $
     \dM^F_{\varrho}[\varepsilon,\varepsilon]_{\smile}=2[\dM^F_{\varrho}\varepsilon,\varepsilon]_{\smile}=-2[\varepsilon,\dM^F_{\varrho}\varepsilon]_{\smile}.
     $
  For any $x, y\in\g$,
        \begin{eqnarray*}
            \delta([\varepsilon,\varepsilon]_{\smile})(x,y)&=&[\varepsilon,\varepsilon]_{\smile}(D(x),y)+[\varepsilon,\varepsilon]_{\smile}(x,D(y))-K([\varepsilon,\varepsilon]_{\smile}(x,y))\\
            %&=&[\varepsilon(D(x)),\varepsilon(y)]_{\h}-[\varepsilon(y),\varepsilon(D(x))]_{\h}+[\varepsilon(x),\varepsilon(D(y))]_{\h}-[\varepsilon(D(y)),\varepsilon(x)]_{\h}\\
            %&&-K([\varepsilon(x),\varepsilon(y)]_{\h})+K([\varepsilon(y),\varepsilon(x)]_{\h})\\
            &=&2[\varepsilon(D(x)),\varepsilon(y)]_{\h}+2[\varepsilon(x),\varepsilon(D(y))]_{\h}
            -[K(\varepsilon(x)),\varepsilon(y)]_{\h}-[\varepsilon(x),K(\varepsilon(y))]_{\h}\\
            &&+[K(\varepsilon(y)),\varepsilon(x)]_{\h}+[\varepsilon(y),K(\varepsilon(x))]_{\h}\\
            %&=&[\varepsilon(D(x))-K(\varepsilon(x)),\varepsilon(y)]_{\h}-[\varepsilon(D(y))-K(\varepsilon(y)),\varepsilon(x)]_{\h}\\
            %&&+[\varepsilon(x),\varepsilon(D(y))-K(\varepsilon(y))]_{\h}-[\varepsilon(y),\varepsilon(D(x))-K(\varepsilon(x))]_{\h}\\
            &=&[\delta(\varepsilon),\varepsilon]_{\smile}(x,y)+[\varepsilon,\delta(\varepsilon)]_{\smile}(x,y).
        \end{eqnarray*}
      By direct calculation, we obtain $[\delta(\varepsilon), \varepsilon]_{\smile}=[\varepsilon, \delta(\varepsilon)]_{\smile}$.
      Hence, $$\delta([\varepsilon,\varepsilon]_{\smile})=[\delta(\varepsilon),\varepsilon]_{\smile}+[\varepsilon,\delta(\varepsilon)]_{\smile}=2[\varepsilon,\delta(\varepsilon)]_{\smile}.$$
      We finish the proof.
    \end{proof}
\begin{lem}
    With the above notations, for any $\varepsilon\in \Hom(\g,\h)$, we have
    \begin{equation}\label{danddelta}
        \delta(\dM^F_{\varrho}\varepsilon)-\dM^F_{\varrho}(\delta\varepsilon)=-[\varepsilon,\chi]_{\smile}.
    \end{equation}
    If $\varrho:\g\lon\Der(\h)$ is a representation of the \LD pair $(\g, D)$ on $(\h, K)$, i.e.,
    $$
     K(\varrho(x)u)=\varrho(D(x))u+\varrho(x)K(u), \quad \forall x\in\g, ~~u\in\h,
    $$
    we obtain
    \begin{equation}\label{dcomdelta}
        \delta(\dM^F_{\varrho}\varepsilon)-\dM^F_{\varrho}(\delta\varepsilon)=0.
    \end{equation}
    \begin{proof}
    For any $x, y\in\g$, by \eqref{1defichi}, we have
        \begin{eqnarray*}
           &&\delta(\dM^F_{\varrho})\varepsilon(x,y)-\dM^F_{\varrho}(\delta\varepsilon)(x,y)\\
&=&\dM^F_{\varrho}\varepsilon(D(x),y)+\dM^F_{\varrho}\varepsilon(x,D(y))-K(\dM^F_{\varrho}\varepsilon(x,y))-\varrho(x)\delta\varepsilon(y)+\varrho(y)\delta\varepsilon(x)+\delta\varepsilon([x,y]_{\g})\\
           %&=&\varrho(D(x))\varepsilon(y)-\varrho(y)\varepsilon(D(x))-\varepsilon([D(x),y]_{\g})+\varrho(x)\varepsilon(D(y))-\varrho(D(y))\varepsilon(x)-\varepsilon([x,D(y)]_{\g})\\
           %&&-K(\varrho(x)\varepsilon(y))+K(\varrho(y)\varepsilon(x))+K(\varepsilon([x,y]_{\g}))-\varrho(x)\varepsilon(D(y))+\varrho(x)K(\varepsilon(y))+\varrho(y)\varepsilon(D(x))\\
          % &&-\varrho(y)K(\varepsilon(x))+\varepsilon(D([x,y]_{\g}))-K(\varepsilon([x,y]_{\g}))\\
           &=&\varrho(D(x))\varepsilon(y)-K(\varrho(x)\varepsilon(y))+\varrho(x)K(\varepsilon(y))-\varrho(D(y))\varepsilon(x)+K(\varrho(y))\varepsilon(x)-\varrho(y)K(\varepsilon(x))\\
           &=&-[\chi(x),\varepsilon(y)]_{\h}+[\chi(y),\varepsilon(x)]_{\h}\\
           &=&-[\varepsilon,\chi]_{\smile}(x,y),
        \end{eqnarray*}
        which implies $\delta(\dM^F_{\varrho}\varepsilon)-\dM^F_{\varrho}(\delta\varepsilon)=-[\varepsilon,\chi]_{\smile}$. If $\varrho:\g\lon\Der(\h)$ is a representation of the \LD pair $(\g, D)$ on $(\h, K)$, we obtain
        \begin{eqnarray*}
            &&\delta(\dM^F_{\varrho})\varepsilon(x,y)-\dM^F_{\varrho}(\delta\varepsilon)(x,y)\\
            &=&\varrho(D(x))\varepsilon(y)-K(\varrho(x)\varepsilon(y))+\varrho(x)K(\varepsilon(y))-\varrho(D(y))\varepsilon(x)+K(\varrho(y)\varepsilon(x))-\varrho(y)K(\varepsilon(x))=0,
        \end{eqnarray*}
        for all $x,y\in\g$. 
    \end{proof}
\end{lem}

Let $(\g, D)$ and $(\h, K)$ be \LD pairs and $\frkk:\g\lon\Out(\h)$ be a $(\g,D)$-kernel for $(\h,K)$. By Proposition \ref{repkk}, the linear map $\rho_\frkk:\g\lon\gl(\frkz(\h))$ defined by
$$
\rho_\frkk(x)u=\frkk(x)(u), \quad \forall x\in\g, u\in\frkz(\h),
$$
is a representation of $(\g, D)$ on $(\frkz(\h), K)$. Denote by $Z^{n}_{\LD}(\g,D; \frkz(\h), K)_\frkk$ the set of $n$-th cocycles of $(\g, D)$ with coefficients in the representation $(\rho_\frkk, K)$, and $H^{n}_{\LD}(\g, D;\frkz(\h), K)_{\frkk}$ the $n$-th cohomology group. Choose a linear map $\varrho:\g\lon\Der(\h)$ such that $\overline{\varrho}=\frkk$.
\begin{pro}\label{omegachiclosed}
With the above notations, $(\dM^F_{\varrho}\omega, \dM^F_{\varrho}\chi+\delta\omega)\in Z^3_{\LD}(\g,D;\frkz(\h),K)_{\frkk}$, i.e., $$\partial_{\rho_\frkk}(\dM^F_{\varrho}\omega, \dM^F_{\varrho}\chi+\delta\omega)=0,$$
where $\omega$ and $\chi$ given by \eqref{1defiomega} and \eqref{1defichi}.
\end{pro}
 \begin{proof}
     See the Appendix A: Proof of Proposition \ref{omegachiclosed}.
 \end{proof}  

\begin{thm}\label{defich}
    The cohomology class $\mathrm{ch}(\frkk):=[(\dM^F_{\varrho}\omega,\dM^F_{\varrho}\chi+\delta\omega)]\in H^3(\g,D;\frkz(\h),K)_{\frkk}$ is well defined.
    \end{thm}
    \begin{proof}
        If $\varrho$ is fixed and $(\omega',\chi')\in C^2_{\LD}(\g,\h)$ also satisfies 
\begin{eqnarray*}
    \ad \omega'(x,y)=[\varrho(x),\varrho(y)]_{\Der(\h)}-\varrho([x,y]_{\g}),\quad
  \ad\chi'(x)=K\circ\varrho(x)-\varrho(D(x))-\varrho(x)\circ K,
\end{eqnarray*}
then $(\omega'-\omega,\chi'-\chi)\in C^2_{\LD}(\g,\frkz(\h))_{\frkk}$ as shown in the proof of Theorem \ref{gdk}. Thus, we have
        \begin{eqnarray*}
(\dM^F_{\varrho}\omega',\dM^F_{\varrho}\chi'+\delta\omega')
       % &=&\Big(\dM^F_{\varrho}\omega+\dM^F_{\varrho}(\omega'-\omega),\dM^F_{\varrho}\chi+\delta\omega+\dM^F_{\varrho}(\chi'-\chi)+\delta(\omega'-\omega)\Big)\\
        %&=&(\dM^F_{\varrho}\omega,\dM^F_{\varrho}\chi+\delta\omega)+(\dM^{F}_{\varrho}(\omega'-\omega),\dM^{F}_{\varrho}(\chi'-\chi)+\delta(\omega'-\omega))\\
        &=&(\dM^F_{\varrho}\omega,\dM^F_{\varrho}\chi+\delta\omega)+(\dM^{CE}_{\rho_\frkk}(\omega'-\omega),\dM^{CE}_{\rho_\frkk}(\chi'-\chi)+\delta(\omega'-\omega))\\
        &=&(\dM^F_{\varrho}\omega,\dM^F_{\varrho}\chi+\delta\omega)+\partial_{\rho_\frkk}(\omega'-\omega,\chi'-\chi).
        \end{eqnarray*}
        Hence, the cohomology class $\mathrm{ch}(\frkk)$ is independent of the choice of the pair $(\omega,\chi)$.

        If $\varrho':\g\lon\Der(\h)$ is another linear map with $\overline{\varrho'(x)}=\frkk(x),$ for all $x$ in $\g$, then $$\varrho'=\varrho+\ad\circ r,$$ for some linear map $r:\g\lon\h$. Define $\omega':\wedge^{2}\g\lon\h$ and $\chi':\g\lon\h$ by
        \begin{equation*}
            \omega'=\omega+\dM^F_{\varrho}r+\frac{1}{2}[r,r]_{\smile}, \quad \quad \chi'=\chi-\delta(r).
        \end{equation*}
        By direct calculation, we obtain equations
        \begin{eqnarray*}
            \ad(\omega'(x,y))=[\varrho'(x),\varrho'(y)]_{\Der(\h)}-\varrho'([x,y]_{\g}),\quad 
            \ad(\chi'(x))=K\circ\varrho'(x)-\varrho'(D(x))-\varrho'(x)\circ K.
        \end{eqnarray*}
         By (\ref{dsmile}) and Proposition \ref{omegachiclosed}, we obtain
        \begin{eqnarray*}
            \dM^F_{\varrho'}\omega'&=&\dM^F_{\varrho}\omega'+[r,\omega']_{\smile}%=\dM^F_{\varrho}\omega+(\dM^F_{\varrho})^2r+\frac{1}{2}\dM^F_{\varrho}[r,r]_{\smile}-[\omega',r]_{\smile}\\
            =\dM^F_{\varrho}\omega+(\dM^F_{\varrho})^{2}r+[\dM^F_{\varrho}r,r]_{\smile}-[\omega',r]_{\smile}.
        \end{eqnarray*}
        Furthermore, by \cite[Lemma 7.6.17]{HN}, we have
        $(\dM^F_{\varrho})^2r(x,y,z)=[\omega,r]_{\smile}(x,y,z)$ for all $x,y,z\in\g$. Thus, by \eqref{rrr0}, it follows that
        \begin{equation*}
            \dM^F_{\varrho'}\omega'=\dM^F_{\varrho}\omega+[\omega+\dM^F_{\varrho}r-\omega',r]_{\smile}=\dM^F_{\varrho}\omega-\frac{1}{2}[[r,r]_{\smile},r]_{\smile}=\dM^F_{\varrho}\omega.
        \end{equation*}
        By (\ref{deltasmile}) and (\ref{danddelta}),
        \begin{eqnarray*}
            \dM^F_{\varrho'}\chi'+\delta\omega'%&=&%\dM^F_{\varrho}\chi'+[r,\chi']_{\smile}+\delta\omega'
            =\dM^F_{\varrho}\chi-\dM^F_{\varrho}\delta r+[r,\chi]_{\smile}-[r,\delta r]_{\smile}+\delta\omega+\delta\dM^F_{\varrho}r+\frac{1}{2}\delta[r,r]_{\smile}
           % &=&\dM^F_{\varrho}\chi+\delta\omega+[r,\chi]_{\smile}+\delta\dM^F_{\varrho}r-\dM^F_{\varrho}\delta r-[r,\delta r]_{\smile}+\frac{1}{2}\delta[r,r]_{\smile}\\
            =\dM^F_{\varrho}\chi+\delta\omega,
        \end{eqnarray*}
        which implies that $(\dM^F_{\varrho'}\omega',\dM^F_{\varrho'}\chi'+\delta\omega')=(\dM^F_{\varrho}\omega,\dM^F_{\varrho}\chi+\delta\omega)$. 
     \end{proof}

\begin{thm}\label{classofkernel}
  Let $(\g, D)$ and $(\h, K)$ be \LD pairs. For a $(\g,D)$-kernel $\frkk:\g\lon\Out(\h)$, the set $\rm{Ext}(\g,D;\h,K)_{\frkk}$ is nonempty if and only if the cohomology class ${\rm ch}(\frkk)\in H^3_{\LD}(\g,D;\frkz(\h),K)_{\frkk}$ vanishes.
    \begin{proof}
        If the set $\rm{Ext}(\g,D;\h,K)_{\frkk}$ is nonempty, then there exists a non-abelian extension of $(\g, D)$ by $(\h, K)$ such that $\frkk$ is the corresponding $(\g, D)$-kernel. By Theorem \ref{corres} and Proposition \ref{barrho}, there exists a non-abelian 2-cocycle $(\varrho, \omega, \chi)$ such that $\bar{\varrho}=\frkk$. By \eqref{nc2} and \eqref{nc4}, we have
$(\dM^F_{\varrho}\omega,\dM^F_{\varrho}\chi+\delta\omega)=0$. Then ${\rm ch}(\frkk)=[(\dM^F_{\varrho}\omega,\dM^F_{\varrho}\chi+\delta\omega)]=0$.

        Since $(\g, D)$ and $(\h, K)$ are \LD pairs and $\frkk:\g\lon\Out(\h)$ is a $(\g,D)$-kernel for $(\g,K)$, one can choose a linear map $\varrho:\g\lon\Der(\h)$ such that $\overline{\varrho}=\frkk$ and ${\rm ch}(\frkk)=[(\dM^F_{\varrho}\omega,\dM^F_{\varrho}\chi+\delta\omega)]$, where $\omega$ and $\chi$ given by \eqref{1defiomega} and \eqref{1defichi}. If ${\rm ch}(\frkk)=0\in H^3(\g,D;\frkk(\h),K)_{\frkk}$, then there exist linear maps $\eta\in C^2(\g,\frkz(\h))_{\frkk}$ and $\theta\in C^1(\g,\frkz(\h))_{\frkk}$, such that $(\dM^{CE}_{\rho_{\frkk}}\eta,\dM^{CE}_{\rho_{\frkk}}\theta+\delta\eta)=\partial_{\rho_\frkk}(\eta,\theta)=(\dM^F_{\varrho}\omega,\dM^F_{\varrho}\chi+\delta\omega)$. Define
        $\omega'=\omega-\eta, \  \chi'=\chi-\theta$, it follows that 
        \begin{equation*}
            \ad(\omega'(x,y))=\ad(\omega(x,y))=[\varrho(x),\varrho(y)]_{\Der(\h)}-\varrho[x,y]_{\g},
        \end{equation*}
        and
        \begin{equation*}
            \ad\chi'=\ad\chi=K\circ\varrho(x)-\varrho(Dx)-\varrho(x)\circ K,
        \end{equation*}
        for all $x,y\in\g$. Furthermore, we obtain
        \begin{equation*}
            \dM^F_{\varrho}\omega'=\dM^F_{\varrho}\omega-\dM^F_{\varrho}\eta=\dM^F_{\varrho}\omega-\dM^{CE}_{\rho_{\frkk}}\eta=0,
        \end{equation*}
        and
        \begin{equation*}
            \dM^F_{\varrho}\chi'+\delta\omega'=\dM^F_{\varrho}\chi-\dM^F_{\varrho}\theta+\delta\omega-\delta\eta=\dM^F_{\varrho}\chi+\delta\omega-(\dM^{CE}_{\rho_{\frkk}}\theta+\delta\eta)=0.
        \end{equation*}
        Hence, $(\varrho,\omega'\chi')$ is a non-abelian 2-cocycle. By Theorem \ref{corres},  there exists a non-abelian extension of $(\g, D)$ by $(\h, K)$ corresponding to $\frkk$.
    \end{proof}
\end{thm}

\section{Applications: Extensibility of derivations of Lie algebras and the obstruction class}
In this section, we use the non-abelian extensions of Lie algebras with derivations to study the following question: 
\begin{itemize}
    \item
Given a non-abelian extension of Lie algebras $0\lon\h\lon\hat{\g}\lon\g\lon0$, let $(K,D)\in\Der(\h)\times\Der(\g)$ be a pair of derivations of $\h$ and $\g$ respectively. When does there exist a derivation $\hat{D}$ of $\hat{\g}$ such that $$\hat{D}|_\h=K, \quad \text{and}\quad  D\circ p=p\circ\hat{D}.$$ 
\end{itemize}
If such a $\hat{D}$ exists, we say that the pair of derivations $(K,D)$  is {\bf extensible}. 

Let $0\lon\h\lon\hat{\g}\lon\g\lon0$ be a non-abelian extension of $\g$ by $\h$. Denote by $\Der_{\h}(\hat{\g})$ the set of all the derivations of $\hat{\g}$ that preserve $\h$ i.e., those $\hat{D}\in\Der(\hat{\g})$ for which $\hat{D}|_\h\in\Der(\h)$. 
Choose a section $s:\g\lon\hat{\g}$ of the non-abelian extension, then we can define a map
$\Gamma:\Der_{\h}(\hat{\g})\lon\Der(\h)\times\gl(\g)$ by
$$
\Gamma(\hat{D})=(\hat{D}|_\h, p\circ \hat{D}\circ s), \quad \forall~~\hat{D}\in\Der_{\h}(\hat{\g}),
$$
where $p:\hat{\g}\lon\g$ is the projection. Note that for two different sections $s:\g\lon\hat{\g}$ and $s':\g\lon\hat{\g}$, we have 
$$
p\circ\hat{D}\circ(s-s')=0,
$$
then the definition of $\Gamma$ is independent on the choice of the section $s$. Denote by {\rm Im}$(\Gamma)$ the image of $\Gamma$.
\begin{pro}
With the above notations, we have that {\rm Im}$(\Gamma)\subset\Der(\h)\times\Der(\g)$ and $\Gamma$ is a Lie algebra homomorphism. In other words, $\Gamma:\Der_{\h}(\hat{\g})\lon\Der(\h)\times\Der(\g)$ is a Lie algebra homomorphism.
\end{pro}
\begin{proof}
For all $x, y\in\g$, since 
\begin{eqnarray*}
&&[(p\circ\hat{D}\circ s)(x),y]_{\g}+[x,(p\circ\hat{D}\circ s)(y)]_{\g}-(p\circ\hat{D}\circ s)([x,y]_{\g})\\
%&=&[(p\circ\hat{D}\circ s)(x),(p\circ s)(y)]_{\g}+[(p\circ s)(x),(p\circ\hat{D}\circ s)(y)]_{\g}-(p\circ\hat{D}\circ s)([x,y]_{\g})\\
%&=&p([\hat{D}(s(x)),s(y)]_{\g}+[s(x),\hat{D}(s(y))]_{\g})-(p\circ\hat{D}\circ s)([x,y]_{\g})\\
&=&p(\hat{D}[(s(x)),s(y)]_{\g})-p\hat{D}(s([x,y]_{\g}))=p\hat{D}([s(x),s(y)]_{\hat{\g}}-s([x,y]_{\g}))=0,
\end{eqnarray*}
then $p\circ \hat{D}\circ s\in\Der(\g)$. 

Since
$
   p\big( \hat{D}(s(x))-s(p\circ\hat{D}\circ s(x))\big)=0,%=p\hat{D}s(x)-ps(p\hat{D}s(x))=p\hat{D}s(x)-p\hat{D}s(x)=0,
$ it follows that $\mathrm{Im}\big((\Id-s\circ p)\circ\hat{D}\circ s\big)\in\h$. Then
\begin{equation*}
    p\circ\hat{D}\circ\hat{D}'\circ s-(p\circ\hat{D}\circ s)\circ (p\circ\hat{D}'\circ s)=p\circ\hat{D}\circ(\Id-s\circ p)\circ\hat{D}'\circ s=0,
\end{equation*}
for all $\hat{D},\hat{D}'\in\Der_{\h}(\hat{\g})$, which implies that
$
    p\circ[\hat{D},\hat{D}']_{\Der(\hat{\g})}\circ s %&=&p\circ\hat{D}\circ\hat{D'}\circ s-p\circ\hat{D}'\circ\hat{D}\circ s\\
    %&=&(p\circ\hat{D}\circ s)\circ (p\circ\hat{D}'\circ s)-(p\circ\hat{D}'\circ s)\circ (p\circ\hat{D}\circ s)\\
    =[p\circ\hat{D}\circ s,p\circ\hat{D}'\circ s]_{\Der(\g)}$,
and 
\begin{eqnarray*}
    \Gamma([\hat{D},\hat{D}']_{\Der(\hat{\g})})%&=&([\hat{D},\hat{D}']_{\Der(\hat{\g})}|_{\h},p\circ[\hat{D},\hat{D}']_{\Der(\hat{\g})}\circ s)\\
    &=&([\hat{D}|_{\h},\hat{D}'|_{\h}]_{\Der(\h)},[p\circ\hat{D}\circ s,p\circ\hat{D}'\circ s]_{\Der(\g)})\\
    %&=&[(\hat{D}|_{\h},p\circ\hat{D}\circ s),(\hat{D}'|_{\h},p\circ\hat{D}'\circ s)]_{\Der(\h)\times\Der(\g)}\\
    &=&[\Gamma(\hat{D}),\Gamma(\hat{D}')]_{\Der(\h)\times\Der(\g)}.
\end{eqnarray*} 
So, the map $\Gamma:\Der_{\h}(\hat{\g})\lon\Der(\h)\times\Der(\g)$ is a Lie algebra homomorphism.
\end{proof}
Note that if $\hat{D}|_\h=K$, taking a section $s:\g\lon\hat{\g}$, we have 
\begin{eqnarray*}
    (p\circ\hat{D}-p\circ\hat{D}\circ s\circ p)(\hat{x})=p\Big(\hat{D}(\hat{x}-sp(\hat{x}))\Big)=p\Big(K(\hat{x}-sp(\hat{x}))\Big)=0,
\end{eqnarray*}
for all $\hat{x}\in\hat{\g}$.
Therefore, finding a derivation $\hat{D}$ that satisfies $D\circ p=p\circ \hat{D}$ is equivalent to taking a section $s:\g\lon\hat{\g}$ and then finding a derivation $\hat{D}$ that satisfies $D=p\circ\hat{D}\circ s$. 

Since $0\lon\h\lon\hat{\g}\lon\g\lon0$ is a non-abelian extension of $\g$ by $\h$, let $s:\g\lon\hat\g$ be a section. By \cite[Proposition 1.1]{Fr},  the pair $(\varrho,\omega)$ induced by $s$ is the corresponding non-abelian $2$-cocycle, where
\begin{eqnarray}\label{non121}
\left\{\begin{array}{rcl}
~~\varrho(x)u&=&[s(x),u]_{\hat{\g}},\quad\forall x\in\g, u\in\h, \\
~~\omega(x,y)&=&[s(x),s(y)]_{\hat{\g}}-s([x,y]_\g), \quad \forall x, y\in\g.\\
%~~\chi(x)&=&\hat{D}(s(x))-s(D(x)), \quad \forall x, y\in\g, u\in\h,
\end{array}\right.
\end{eqnarray}
\begin{thm}\label{inducible111}
With the above notations, a pair $(K,D)\in  \Der(\h)\times \Der(\g)$ lies in {\rm Im}$(\Gamma)$ if and only if there exists a linear map $\chi:\g\to \h$, such that
\begin{eqnarray}
\label{11}[\chi(x),u]_{\h}&=&K(\varrho(x)u)-\varrho(D(x))-\varrho(x)K(u),\\
\label{12}K(\omega(x,y))-\omega(Dx,y)-\omega(x,Dy)&=&\varrho(x)\chi(y)-\varrho(y)\chi(x)-\chi([x,y]_{\g}),
\end{eqnarray}
for all $x,y\in \g, \  u\in \h$. 
\end{thm}
\begin{proof}
If $(K,D)\in\rm{Im}(\Gamma)$, then there exists a derivation $\hat{D}\in \Der_{\h}(\hat{\g})$ such that $\Gamma(\hat{D})=(K,D)$, i.e., $K=\hat{D}|_{\h}$ and $D=p\circ \hat{D}\circ s$. In this case, $(\hat{\g},\hat{D})$ is a non-abelian extension of $(\g,D)$ by $(\h,K)$, i.e., we have the following commutative diagram:
\[\begin{CD}
0@>>>\h@>i>>\hat{\g}@>p>>\g            @>>>0\\
@.    @V K VV   @V\hat{D}VV  @V D VV    @.\\
0@>>>\h @>i>>\g@>p>>\g           @>>>0.
\end{CD}\]
Define $\chi:\g\lon\h$ by $\chi(x)=\hat{D}(s(x))-s(D(x))$ for all $x\in\g$, then \eqref{11} and \eqref{12} hold as $(\varrho, \omega, \chi)$ is a non-abelian $2$-cocycle.

Conversely, suppose that there exists a linear map $\chi:\g\lon \h$ satisfying the \eqref{11} and \eqref{12}. Since any element of $\hat{\g}$ can be written uniquely in the form $s(x)+u$ for $x\in\g,u\in \h$, we define a linear map $\hat{D}:\hat{\g}\lon\hat{\g}$ by
\begin{equation*}
\hat{D}(s(x)+u)= s(D(x))+K(u)+\chi(x), \quad\forall ~~x\in\g, u\in\h.
\end{equation*}
Then the map $\hat{D}$ is a derivation of $\hat{\g}$. Moreover, by the definition of $\hat{D}$, we have
\begin{equation*}
\hat{D}(u)=K(u),\quad p\Big(\hat{D} (s(x))\Big)=p\Big(s(D(x))+\chi(x)\Big)=D(x),
\end{equation*}
for all $x\in \g,\ u\in \h$. In other words, $(K,D)=\Gamma(\hat{D})$, i.e., $(K,D)\in\mathrm{Im}(\Gamma)$.
\end{proof}
Finally, we give the obstruction class of the extensibility of derivations of Lie algebras.
\begin{defi}
    Let $(K,D)\in\Der(\h)\times\Der(\g)$ be a pair of derivations, if there exists a linear map $\chi:\g\lon\h$ such that
    \begin{equation}\label{deficom}
        [\chi(x),u]_{\h}=K(\varrho(x)u)-\varrho(D(x))-\varrho(x)K(u),
    \end{equation}
    for all $x\in\g,\ u\in\h$. Then, $K$ and $D$ are called {\bf compatible}. Denote the set of compatible derivations of $\g$ and $\h$ by $C\Der(\h,\g)$.
\end{defi}
Let $0\lon\h\lon\hat{\g}\lon\g\lon0$ be a non-abelian extension of $\g$ by $\h$. Choose a section $s:\g\lon\hat\g$, and let $(\varrho,\omega)$ be the corresponding non-abelian $2$-cocycle induced by $s$, as in \eqref{non121}. Then we have that $\dM^F_{\varrho}\omega=0$ and $(\varrho, \omega)$ satisfies \eqref{nc1} according to \cite[Proposition 1.1]{Fr}. Then $\varrho:\g\lon\gl(\frkz(\h))$
is a representation of $\g$ on $\frkz(\h)$. Assume that $K$ and $D$ are compatible, then there exists $\chi:\g\lon\h$ satisfying \eqref{deficom}. By \eqref{deficeo}, Proposition \ref{omegachiclosed} and Theorem \ref{defich}, we have that $[\dM^F_\varrho\chi+\delta(\omega)]\in H^2(\g, \frkz(\h))$, where $H^{2}(\g, \frkz(\h))$ is the second cohomology group of the Lie algebra $\g$ with coefficients in the representation $(\frkz(\h), \varrho).$ Thus we obtain a map  
    \begin{equation*}
        \huaW:C\Der(\h,\g)\lon H^2(\g,\frkz(\h)),\quad (K,D)\mapsto [\dM^F_{\varrho}\chi+\delta(\omega)].
    \end{equation*}
\begin{thm}
With the above notations, a pair of derivations $(K, D)$ is extensible if and only if $K$ and $D$ are compatible and $\huaW(K, D)=0$. The cohomology class $\huaW(K, D)$ is called the obstruction class for the extensibility of $(K, D)$. 
    \begin{proof}
        If $(K, D)$ is extensible, then there exists a derivation $\hat{D}\in\Der(\hat{\g})$, such that $\Gamma(\hat{D})=(K,D)$. 
        Thus, we obtain that $(\hat{\g},\hat{D})$ is a \LD pair non-abelian extension of $(\g,D)$ by $(\h,K)$. 
        Define $\chi:\g\lon\h$ by
        $$
        \chi(x)=\hat{D}(s(x))-s(D(x)), \quad \forall x\in\g.
        $$
        Then $(\varrho, \omega, \chi)$ is a non-abelian $2$-cocycle of $(\g, D)$ with values in $(\h, K)$, which means that $K$ and $D$ are compatible. By Theorem \ref{classofkernel}, we get that $[\dM^F_{\varrho}\chi+\delta(\omega)]=0\in H^2(\g,\frkz(\h))$.

        If there exists a linear map $\chi:\g\lon\h$, such that
       $$
       [\chi(x),u]_{\h}=K(\varrho(x)u)-\varrho(D(x))-\varrho(x)K(u), \quad\forall x\in\g, u\in\h,
       $$
       and $[\dM^F_{\varrho}\chi+\delta(\omega)]=0\in H^2(\g,\frkz(\h))$. It follows that there exists a linear map $\theta\in \Hom(\g,\frkz(\h))$, such that $\dM^F_{\varrho}\chi+\delta(\omega)=\dM^{CE}_{\varrho}\theta$. \emptycomment{i.e.,
        \begin{equation*}
            \varrho(x)\chi(y)-\varrho(y)\chi(x)-\chi([x, y]_\g)+\omega(D(x), y)+\omega(x, D(y))-K(\omega(x, y))=\varrho(x)\theta(y)-\varrho(y)\theta(x)-\theta([x, y]_\g),
        \end{equation*}
        for all $x, y\in\g$.} Define $\chi'=\chi-\theta$, we have
        \begin{equation*}
          [\chi'(x),u]_{\h}=K(\varrho(x)u)-\varrho(D(x))-\varrho(x)K(u), \quad \forall x\in\g, u\in\h,
        \end{equation*}
        and
        \begin{equation*}
            \varrho(x)\chi'(y)-\varrho(y)\chi'(x)-\chi'([x, y]_\g)=K(\omega(x, y))-\omega(D(x), y)-\omega(x, D(y)),
        \end{equation*}
        for all $x, y\in\g$. Thus, $\chi'$ satisfies equations (\ref{11}) and (\ref{12}). By Proposition \ref{inducible111}, the pair $(K,D)$ is extensible.
    \end{proof}
\end{thm}
\begin{rmk}
	From the proof of the above theorem, we can obtain $\ker\huaW=\mathrm{Im}(\Gamma)$. Furthermore, define a linear map $i: Z^{1}(\g,\frkz(\h))_{\frkk}\lon \Der_{\h}(\hat{\g})$ by $$
	i(\chi)(X)=\chi(s(X)), \quad \forall X\in\hat{\g}.
	$$
	By direct calculation, we have $\mathrm{Im}(i)=\ker\Gamma$. 
	Thus, there is a {\rm Wells} type exact sequence in the context of extension of derivations of Lie algebras as follows:
	$$
	0\longrightarrow Z^{1}(\g,\frkz(\h))_{\frkk}\overset{i}\longrightarrow \Der_{\h}(\hat{\g}) \overset{\Gamma}\longrightarrow C\Der(\h,\g)_{\frkk}\overset{\huaW}\longrightarrow H^2(\g,\frkz(\h))_{\frkk}.
	$$
	\end{rmk}
\begin{rmk}
When the extension of $\g$ by $\h$ reduces to a central extension, all pairs of derivations $(K,D)\in\Der(\h)\times\Der(\g)$ are compatible. Thus, the linear map $\chi$ in the definition of $\huaW$ can be taken as $0$. Hence, we have $\huaW(K,D)=[\delta(\omega)]$, which is the obstruction class defined in \cite{TFS}.
\end{rmk}

\section*{Appendix A: Proof of Proposition \ref{omegachiclosed}}
In this appendix, we give a proof of Proposition \ref{omegachiclosed}. First, we need the following result.

\begin{pro}
Let $\frkk:\g\lon\Out(\h)$ be a $(\g,D)$-kernel for $(\h,K)$. Define a linear space $\tilde{\g}$ by
\begin{equation*}
    \tilde{\g}=\frkk^*\Der(\h)=\{(x,K')\in \g\times\Der(\h): \overline{K'}=\frkk(x)\}.
\end{equation*} 
Then $(\tilde{\g}, [\cdot, \cdot])$ is a Lie subalgebra of the direct product Lie algebra $(\g\times\Der(\h), [\cdot, \cdot])$. Moreover, define a linear map $\tilde{D}:\tilde{\g}\lon\tilde{\g}$ by
$$\widetilde{D}(x,K')=(D(x),[K,K']_{\Der(\h)}), \quad \forall (x, K')\in\tilde{\g}.$$ Then $(\tilde{\g}, \tilde{D})$ is a \LD pair.
\end{pro}
\begin{proof}
Let $(x, K')$ and $(y, K'')$ in $\tilde\g$. Then $\frkk(x)=\overline{K'}$ and $\frkk(y)=\overline{K''}$, which implies
\begin{equation*}
    \overline{[K,K']_{\Der(\h)}}=[\overline{K},\overline{K'}]_{\Out(\h)}=[\frkk(x),\frkk(y)]_{\Out(\h)}=\frkk([x, y]_\g).
\end{equation*}
Thus $(\tilde{\g}, [\cdot, \cdot])$ is a Lie subalgebra of the direct product Lie algebra $(\g\times\Der(\h), [\cdot, \cdot])$.

Let $(x, K')\in\tilde{\g}$. Then $\frkk(x)=\overline{K'}$. By \eqref{gdk0}, we have
\begin{equation*}
    \overline{[K,K']_{\Der(\h)}}=[\overline{K},\overline{K'}]_{\Out(\h)}=[\overline{K},\frkk(x)]_{\Out(\h)}=\frkk(D(x)),
\end{equation*}
then $(D(x), [K, K']_{\Der(\h)})\in\g\times\Der(\h)$. Thus, 
the linear map $\widetilde{D}:\tilde{\g}\lon\tilde{\g}$ is well defined.
Moreover, for any $(x, K')\in\tilde{\g}$ and $(y, K'')\in\tilde{\g}$,
\begin{eqnarray*}
    \widetilde{D}([(x,K'),(y,K')])
    &=&\widetilde{D}([x,y]_{\g},[K',K'']_{\Der(\g)})\\%=(D([x,y]_{\g}),[K,[K',K'']_{\Der(\h)}]_{\Der(\h)})\\
    %&=&([D(x),y]_{\g}+[x,D(y)]_{\g},[[K,K']_{\Der(\g)},K'']_{\Der(\h)}+[K',[K,K'']_{\Der(\h)}]_{\Der(\h)})\\
   % &=&([D(x),y]_{\g},[[K,K']_{\Der(\g)},K'']_{\Der(\h)})+([x,D(y)]_{\g},[K',[K,K'']_{\Der(\h)}]_{\Der(\h)})\\
    &=&[(D(x),[K,K']_{\Der(\h)}),(y,K'')]+[(x,K'),(D(y),[K,K'']_{\Der(\h)})]\\
    &=&[\widetilde{D}(x,K'),(y,K'')]+[(x,K'),\widetilde{D}(y,K'')],
\end{eqnarray*}
which means that $\widetilde{D}:\tilde{\g}\lon\tilde{\g}$ is a derivation.
Therefore $(\tilde{\g}, \widetilde{D})$ is a \LD pair.
\end{proof}

Finally, we give the proof of Proposition \ref{omegachiclosed}.
\begin{proof}
    Since $\ad \omega(x,y)=[\varrho(x),\varrho(y)]_{\Der(\h)}-\varrho([x,y]_{\g})$ for all $x, y\in\g$, we have
    \begin{eqnarray*}
        &&\ad(\dM^F_{\varrho}\omega(x.y, z))\\
        %&=&\ad(\varrho(x)\omega(y,z))-\ad(\omega([x,y]_{\g},z))+\ad(\varrho(y)\omega(z,x))-\ad(\omega([y,z]_{\g},x))\\
        %&&+\ad(\varrho(z)\omega(x,y))-\ad(\omega([z,x]_{\g},y))\\
        &=&[\varrho(x),\ad(\omega(y,z))]_{\Der(\h)}-[\varrho([x,y]_{\g}),\varrho(z)]_{\Der(\h)}+\varrho([[x,y]_{\g},z]_{\g})+[\varrho(y),\ad(\omega(z,x))]_{\Der(\h)}\\
        &&-[\varrho([y,z]_{\g}),\varrho(x)]_{\Der(\h)}+\varrho([[y,z]_{\g},x]_{\g})+[\varrho(z),\ad(\omega(x,y))]_{\Der(\h)}-[\varrho([z,x]_{\g}),\varrho(y)]_{\Der(\h)}\\
        &&+\varrho([[z,x]_{\g},y]_{\g})\\
       % &=&[\varrho(x),[\varrho(y),\varrho(z)]_{\Der(\h)}]_{\Der(\h)}-[\varrho(x),\varrho([y, z]_{\g})]_{\Der(\h)}-[\varrho([x,y]_{\g}),\varrho(z)]_{\Der(\h)}\\
        %&&+\varrho([[x,y]_{\g},z]_{\g})+[\varrho(y),[\varrho(z),\varrho(x)]_{\Der(\h)}]_{\Der(\h)}-[\varrho(y),\varrho([z, x]_{\g})]_{\Der(\h)}\\
       % &&-[\varrho([y,z]_{\g}),\varrho(x)]_{\Der(\h)}+\varrho([[y,z]_{\g},x]_{\g})+[\varrho(z),[\varrho(x),\varrho(y)]_{\Der(\h)}]_{\Der(\h)}\\
        %&&-[\varrho(z),\varrho([x,y]_{\g})]_{\Der(\h)}-[\varrho([z,x]_{\g}),\varrho(y)]_{\Der(\h)}+\varrho([[z,x]_{\g},y]_{\g})\\
        &=&[\varrho(x),[\varrho(y),\varrho(z)]_{\Der(\h)}]_{\Der(\h)}+\varrho([[x,y]_{\g},z]_{\g})+[\varrho(y),[\varrho(z),\varrho(x)]_{\Der(\h)}]_{\Der(\h)}+\varrho([[y,z]_{\g},x]_{\g})\\
        &&+[\varrho(z),[\varrho(x),\varrho(y)]_{\Der(\h)}]_{\Der(\h)}+\varrho([[z,x]_{\g},y]_{\g})=0,
    \end{eqnarray*}
    for all $x,y,z\in\g$, which implies $\dM^F_{\varrho}\omega\in C^3(\g,\frkz(\h))$. Since $$\ad\chi(x)=K\circ\varrho(x)-\varrho(D(x))-\varrho(x)\circ K,\quad\forall x\in\g,$$  it follows that
    \begin{eqnarray*}
         &&\ad (\dM^F_{\varrho}\chi(x,y))\\%&=&\ad(\varrho(x)\chi(y))-\ad(\varrho(y)\chi(x))-\ad(\chi[x,y]_{\g})\\
         &=&[\varrho(x),\ad\chi(y)]_{\Der(\h)}-[\varrho(y),\ad\chi(x)]_{\Der(\h)}-\ad\chi([x,y]_{\g})\\
         &=&[\varrho(x),K\circ\varrho(y)-\varrho(D(y))-\varrho(y)\circ K]_{\Der(\h)}-[\varrho(y),K\circ\varrho(x)-\varrho(D(x))-\varrho(x)\circ K]_{\Der(\h)}\\
         &&-K\circ\varrho([x,y]_{\g})+\varrho([D(x),y]_{\g}+[x,D(y)]_{\g})+\varrho([x,y]_{\g})\circ K.
    \end{eqnarray*}
   Moreover,
    \begin{eqnarray*}
        \ad (\delta\omega(x,y))%&=&\ad(\omega(D(x),y))+\ad(\omega(x,D(y)))-\ad (K(\omega(x,y)))\\
        %&=&\ad(\omega(D(x),y))+\ad(\omega(x,D(y)))-[K,\ad\omega(x,y)]_{\Der(\h)}\\
        &=&[\varrho (D(x)),\varrho (y)]_{\Der(\h)}-\varrho([D(x),y]_{\g})+[\varrho(x),\varrho(D(y))]_{\Der(\h)}-\varrho([x,D(y)]_{\g})\\
        &&-[K,[\varrho (x),\varrho (y)]_{\Der(\h)}]_{\Der(\h)}+[K,\varrho([x,y]_{\g})]_{\Der(\h)}\\
        &=&[\varrho (D(x)),\varrho (y)]_{\Der(\h)}-\varrho([D(x),y]_{\g})+[\varrho(x),\varrho(D(y))]_{\Der(\h)}-\varrho([x,D(y)]_{\g})\\
        &&-[K\circ\varrho(x)-\varrho(x)\circ K,\varrho(y)]_{\Der(\h)}-[\varrho(x), K\circ\varrho(y)-\varrho(y)\circ K]_{\Der(\h)}\\
        &&+K\circ \varrho([x,y]_{\g})-\varrho([x,y]_{\g})\circ K,
    \end{eqnarray*}
    for all $x,y\in\g$. Thus, $\ad (\dM^F_{\varrho}\chi+\delta\omega)=0$, i.e., $\dM^F_{\varrho}\chi+\delta\omega\in C^2(\g,\frkz(\h))$, which means that $$(\dM^F_{\varrho}\omega,\dM^F_{\varrho}\chi+\delta\omega)\in C^3_{\LD}(\g,\frkz(\h)).$$

Since $\overline{\varrho}=\frkk$, by Proposition \ref{repkk}, we have that $\rho_{\frkk}:\g\lon\gl(\frkz(\h))$ is a representation of $(\g,D)$ on $(\frkz(\h),K)$, where
$$
\rho_{\frkk}(x)u=\frkk(x)u=\varrho(x)u, \quad \forall u\in\frkz(\h),  ~~x\in\g.
$$
Denote by $(C_{\LD}^*(\tilde{\g}, \frkz(\h)), \partial_{\rho_\frkk})$ the $n$-th cohomology group of $(\g, D)$ with coefficients in the representation $(\rho_\frkk, K)$. %In the following, we prove $\partial_{\rho_\frkk}(\dM^F_{\varrho}\omega,\dM^F_{\varrho}\chi+\delta\omega)=0$, 
Define a linear map $\tilde{\rho}:\tilde{\g}\lon\Der(\h)$ by $$\tilde{\rho}(x, K')u=K'(u), \quad \forall (x, K')\in\tilde{\g}, ~~u\in\h, $$
since the derivations preserve $\frkz(\h)$, then $\tilde{\rho}$ is a representation of $\tilde{\g}$ on $\frkz(\h)$. Moreover, 
$$
K(\tilde{\rho}(x, K')u)=K(K'(u))=[K, K']_{\Der{\h}}(u)+K'(K(u))=\tilde{\rho}(\tilde{D}(x, K'))u+\tilde{\rho}(x, K')K(u),
$$
which means that $\tilde{\rho}:\tilde{\g}\lon\gl(\frkz(\h))$ is a representation of $(\tilde{\g}, \tilde{D})$ on $(\frkz(\h), K)$. We denote by $(C_{\LD}^*(\tilde{\g}, \frkz(\h)), \partial_{\tilde{\rho}})$ the cochain complex of $(\tilde{\g}, \tilde{D})$ with coefficients in the representation $(\frkz(\h), \tilde{\rho})$.

    Define $q:\tilde{\g}\lon\g$ by $$q(x, K')=x,\quad\forall (x, K')\in\tilde{\g}.$$ By direct calculation, $q^*:(C^*_{\LD}(\g, \frkz(\h)), \partial_{\rho_\frkk})\lon (C_{\LD}^*(\tilde{\g}, \frkz(\h)), \partial_{\tilde{\rho}})$ is a cochain map. 
    
     Define $\tilde{\varrho}:\tilde{\g}\lon\gl(\frkz(\h))$ by 
    $$
    \tilde{\varrho}(x, K')=\varrho(x), \quad\forall (x, K')\in\tilde{\g}.
    $$ 
Denote by $\tilde{\omega}=q^*(\omega)$ and $\tilde{\chi}=q^*(\chi)$, by direct calculation we have \begin{equation*}
        q^*(\dM^{F}_{\varrho}\omega)=\dM^{F}_{\tilde{\varrho}}\tilde{\omega}, \quad \quad q^*(\dM^{F}_{\varrho}\chi+\delta(\omega))=\dM^{F}_{\tilde{\varrho}}\tilde{
        \chi}+\delta(\tilde{\omega}),
    \end{equation*}
    where $\delta$ given by \eqref{deeqsigma}. To prove that $(\dM^F_{\varrho}\omega,\dM^F_{\varrho}\chi+\delta\omega)$ is a 3-cocycle, it suffices to show that the pullback $q^*(\dM^F_{\varrho}\omega,\dM^F_{\varrho}\chi+\delta\omega)=(\dM^{F}_{\tilde{\varrho}}\tilde{\omega}, \dM^{F}_{\tilde{\varrho}}\tilde{\chi}+\delta(\tilde{\omega}))\in C^3_{\LD}(\tilde{\g},\frkz(\h))$ is a $3$-cocycle as $q:\tilde{\g}\lon\g$ is surjective. Since $$\overline{\tilde{\varrho}(x,K')-\tilde{\rho}(x,K')}=\overline{\varrho(x)-K'}=\frkk(x)-\overline{K'}=0,\quad \forall(x,K')\in\tilde{\g},$$ there exists a linear map $\tilde{\varepsilon}:\tilde{\g}\lon\h$ such that 
    \begin{equation}\label{varrhorho}
        \tilde{\varrho}=\tilde{\rho}+\ad\circ\tilde{\varepsilon}:\tilde{\g}\lon\gl(\frkz(\h)).
    \end{equation} 
For any $(x, K')\in\tilde{\g}$ and $u\in\h$, we have
    \begin{eqnarray*}
        \Big(\ad \tilde{\chi}(x,K')\Big)u&=&[\chi(x), u]_\h\\%=K\Big(\tilde{\varrho}(x,K')u\Big)-\tilde{\varrho}(\tilde{D}(x,K'))u-\tilde{\varrho}(x,K')K(u)\\
        &=&K\Big((\tilde{\rho}+\ad\tilde{\varepsilon})(x,K')u\Big)-\Big((\tilde{\rho}+\ad\tilde{\varepsilon})(\tilde{D}(x,K'))\Big)u-(\tilde{\rho}+\ad\tilde{\varepsilon})(x,K')K(u)\\
       % &=&K\Big(\tilde{\rho}(x,K')\Big)u-\tilde{\rho}(\tilde{D}(x,K'))u-\tilde{\rho}(x,K')K(u)\\
       % &&+K(\ad\tilde{\varepsilon}(x,K')u)-\ad\tilde{\varepsilon}(\tilde{D}(x,K'))u-\ad\tilde{\varepsilon}(x,K')K(u)\\
        &=&K(\ad\tilde{\varepsilon}(x,K')u)-\ad\tilde{\varepsilon}(\tilde{D}(x,K'))u-\ad\tilde{\varepsilon}(x,K')K(u)\\
       % &=&\ad K(\tilde{\varepsilon}(x,K'))u-\ad\tilde{\varepsilon}(\tilde{D}(x,K'))u\\
        &=&-\ad\delta(\tilde{\varepsilon}(x, K'))u.
    \end{eqnarray*}
   Then $\tilde{\chi}+\delta(\tilde{\varepsilon})\in\Hom(\tilde{\g}, \frkz(\h))$.
    By (\ref{varrhorho}), we have  
    \begin{eqnarray*}
\dM^{F}_{\tilde{\varrho}}\tilde{\omega}&=&\dM^{CE}_{\tilde{\rho}}\tilde{\omega}+[\tilde{\varepsilon},\tilde{\omega}]_{\smile},
\quad
\dM^{F}_{\tilde{\varrho}}\tilde{\chi}=\dM^{CE}_{\tilde{\rho}}\tilde{\chi}+[\tilde{\varepsilon},\tilde{\chi}]_{\smile}.
    \end{eqnarray*}
    Since for all $(x,K'),(y,K'')\in\tilde{\g}, u\in\h$,
    \begin{eqnarray*}
       && [\tilde{\omega}((x,K'),(y,K''))), u]_\h\\
       %&=&[\omega(x, y), u]_\h=[\varrho(x), \varrho(y)]_{\Der(\h)}(u)-\varrho([x, y]_\g)u\\
        &=&[\tilde{\varrho}(x,K'),\tilde{\varrho}(y,K'')]_{\Der(\h)}(u)-\tilde{\varrho}([(x.K'),(y,K'')]_{\tilde{\g}})u\\
        %&=&[\tilde{\rho}(x, K')+\ad\tilde{\varepsilon}(x, K'),\tilde{\rho}(y, K'')+\ad\tilde{\varepsilon}(y, K'')]_{\Der(\h)}(u)-\tilde{\rho}([(x.K'),(y,K'')]_{\tilde{\g}})u\\
        %&&-\ad\tilde{\varepsilon}([(x.K'),(y,K'')]_{\tilde{\g}})u\\
        &=&\ad\Big(\tilde{\rho}(x,K')\tilde{\varepsilon}(y,K'')-\tilde{\rho}(y,K'')\tilde{\varepsilon}(x,K')-\tilde{\varepsilon}([(x.K'),(y,K'')]_{\tilde{\g}})+[\tilde{\varepsilon}(x,K'),\tilde{\varepsilon}(y,K'')]_{\h}\Big)u\\
&=&\Big(\ad\circ(\dM^{CE}_{\tilde{\rho}}\tilde{\varepsilon}+\frac{1}{2}[\tilde{\varepsilon},\tilde{\varepsilon}]_{\smile})((x,K'),(y,K'))\Big)u,
    \end{eqnarray*}
     which implies that $\tilde{\omega}((x, K'),(y,K''))-(\dM^{CE}_{\tilde{\rho}}\tilde{\varepsilon}+\frac{1}{2}[\tilde{\varepsilon},\tilde{\varepsilon}]_{\smile})((x,K'),(y,K''))\in\frkz(\h).$Furthermore, by \eqref{dsmile} and $[[\tilde{\varepsilon},\tilde{\varepsilon}]_{\smile},\tilde{\varepsilon}]_{\smile}=0$, we obtain
    \begin{equation*}
        [\tilde{\varepsilon},\tilde{\omega}]_{\smile}=\big[\tilde{\varepsilon}, \dM^{CE}_{\tilde{\rho}}\tilde{\varepsilon}+\frac{1}{2}[\tilde{\varepsilon},\tilde{\varepsilon}]_{\smile}\big]_{\smile}=[\tilde{\varepsilon},\dM^{CE}_{\tilde{\rho}}\tilde{\varepsilon}]_{\smile}=-\frac{1}{2}\dM^{CE}_{\tilde{\rho}}[\tilde{\varepsilon},\tilde{\varepsilon}]_{\smile}.
    \end{equation*}
    Hence, $\dM^{F}_{\tilde{\varrho}}\tilde{\omega}=\dM^{CE}_{\tilde{\rho}}\tilde{\omega}+[\tilde{\varepsilon},\tilde{\omega}]_{\smile}=\dM^{CE}_{\tilde{\rho}}\tilde{\omega}-\frac{1}{2}\dM^{CE}_{\tilde{\rho}}[\tilde{\varepsilon},\tilde{\varepsilon}]_{\smile}=\dM^{CE}_{\tilde{\rho}}(\tilde{\omega}-\frac{1}{2}[\tilde{\varepsilon},\tilde{\varepsilon}]_{\smile})$, which means that $\dM^{F}_{\tilde{\varrho}}\tilde{\omega}$ is a coboundary. We have $\dM^{CE}_{\tilde{\rho}}(\dM^{F}_{\tilde{\varrho}}\tilde{\omega})=0$, that is, $\dM^{F}_{\tilde{\varrho}}\tilde{\omega}\in Z^3(\tilde{\g},\frkz(\h))$.
   
    By equations (\ref{deltasmile}) and (\ref{dcomdelta}), we have
    \begin{eqnarray*}
        \dM^{CE}_{\tilde{\rho}}\tilde{\chi}+\delta(\tilde{\omega}-\frac{1}{2}[\tilde{\varepsilon}, \tilde{\varepsilon}]_{\smile})&=&\dM^{CE}_{\tilde{\rho}}\tilde{\chi}+\delta(\tilde{\omega})-[\tilde{\varepsilon}, \delta(\tilde{\varepsilon})]_{\smile}=\dM^{CE}_{\tilde{\rho}}\tilde{\chi}+\delta(\tilde{\omega})+[\tilde{\varepsilon}, \tilde{\chi}]_{\smile}=\dM^{F}_{\tilde{\varrho}}\tilde{\chi}+\delta(\tilde{\omega}),
    \end{eqnarray*}
    which means that $\partial_{\tilde{\rho}}(\tilde{\omega}-\frac{1}{2}[\tilde{\varepsilon}, \tilde{\varepsilon}]_{\smile}, \tilde{\chi})=(\dM^{F}_{\tilde{\varrho}}\tilde{\omega}, \dM^{F}_{\tilde{\varrho}}\tilde{\chi}+\delta(\tilde{\omega}))$. Thus, $(\dM^{F}_{\tilde{\varrho}}\tilde{\omega}, \dM^{F}_{\tilde{\varrho}}\tilde{\chi}+\delta(\tilde{\omega}))$ is coboundary, which implies that $(\dM^{F}_{\tilde{\varrho}}\tilde{\omega}, \dM^{F}_{\tilde{\varrho}}\tilde{\chi}+\delta(\tilde{\omega}))\in Z^3_{\LD}(\tilde{\g},\tilde{D};\frkz(\h),K)_\frkk$. Thus $(\dM^F_{\varrho}\omega,\dM^F_{\varrho}\chi+\delta(\omega))\in Z^3_{\LD}(\g,D;\frkz(\h),K)_{\frkk}$.
\end{proof}

\vspace{2mm}
\noindent
{\bf Acknowledgements. } This research is supported by NSFC (12401076).

%\bibliographystyle{amsplain}
%\bibliography{referencelieder}

 \end{document}